\documentclass{article} 
\usepackage[usenames,dvipsnames]{color}
\usepackage{amsmath} 
\usepackage{amssymb} 
\usepackage{bm}
\usepackage{amsthm}
\usepackage{indentfirst}
\usepackage{amsfonts,mathrsfs}
\usepackage{hyperref}
\numberwithin{equation}{section}
\usepackage{geometry}
\usepackage{esint} 
\newtheorem{theorem}{Theorem}
\newtheorem{lemma}{Lemma}

\newtheorem{corollary}{Corollary}
\newtheorem{definition}{Definition}
\newtheorem{remark}{Remark}

\def\NN{\mathbb{N}}

\def\RR{\mathbb{R}}
\def\QQ{\mathbb{Q}}

\def\ZZ{\mathbb{Z}}

\def\CC{\mathbb{C}}

\def\SS{\mathbb{S}}

\def\mal{\max\limits}

\def\lt{\left}
\def\rt{\right}

\def\mI{\mathcal{I}}

\def\bomega{{\bm\omega}}
\def\balpha{{\bm\alpha}}
\def\bbeta{{\bm\beta}}
\def\bx{{\bm x}}
\def\bj{{\bm j}}

\title{Sharp Sobolev Approximation on General Domains by\\ Linearized Shallow Networks with Analytic Activations}
\author{Jia Li\thanks{Peking University, Beijing 100871, P. R. China}\and Tong Mao\thanks{Shenzhen Loop Area Institute, Shenzhen 518000, P. R. China}\and Jinchao Xu\thanks{King Abdullah University of Science and Technology, Thuwal 23955, Saudi Arabia}}
\date{}

\begin{document}
\maketitle

\begin{abstract}
We study Sobolev approximation on bounded domains by linearized shallow neural networks whose inner parameters are prescribed independently of the target function. Our main step is a one-dimensional construction for analytic activations. We prove that quasi-Chebyshev parameter sets with univariate resolution $m$ generate fixed feature spaces attaining the sharp $H^r$-to-$H^s$ approximation order $m^{-(r-s)}$ for a class of analytic activations satisfying a quantitative non-cancellation condition on their Taylor coefficients. Combining this result with the ridge-function lifting theorem in \cite{petrushev1998approximation} and its extension to arbitrary quasi-uniform direction sets established in this work, we construct tensor-product-type parameter sets that attain the sharp rate
$$\|f-f_n\|_{L^2(\Omega)}\lesssim n^{-\frac rd}\|f\|_{H^r(\Omega)},\quad f\in H^r(\Omega)$$
for all $r>0$. In contrast to the finite-difference construction in \cite{mhaskar1996neural}, whose explicit admissibility condition may require an extremely small parameter scale, the proposed parameter sets remain distributed over fixed intervals and are therefore more amenable to practical computation.
\end{abstract}

\section{Introduction}\label{sec:introduction}

Let $\sigma\colon\RR\to\RR$ be an activation function. A shallow neural network with width $n$ is a function of the form
\begin{equation}\label{eqn:shallow-network-class}
    \Sigma_{n}^\sigma(\Omega)
    =
    \left\{
        x\mapsto\sum_{j=1}^n a_j \sigma\big(w_j\cdot x + b_j\big)
        :
        w_j \in \mathbb R^d,\;
        a_j,b_j \in \mathbb R
    \right\}.
\end{equation}
Classical universal approximation theorems show that networks of this form are dense in broad classes of functions for sigmoidal or, more generally, nonpolynomial activations \cite{cybenko1989approximation,hornik1989multilayer,Hornik1991,leshno1993multilayer}; see also \cite{pinkus1999approximation,devore2021neural}. Quantitative approximation theory asks instead how the best error
\begin{equation*}
    \inf_{g\in\Sigma_n^\sigma}\|f-g\|_X
\end{equation*}
depends on $n$ for a prescribed target class and norm $X$. Fourier-moment conditions and their later function-space formulations give the dimension-independent Monte Carlo rate $O(n^{-1/2})$ for Barron- or variation-type classes \cite{barron1993universal,barron1994approximation,Jones1992,makovoz1996random,devore1996some,bach2017breaking,E2019barron,E2020representation,siegel2023characterization}. Sharp rates, metric entropy, and $n$-widths for several variation spaces have subsequently been developed in \cite{siegel2020approximation,siegel2022high,ma2022uniform,siegel2022sharp}. For classical smoothness classes, sufficiently smooth and nondegenerate activations attain the optimal Sobolev order $O(n^{-r/d})$ under the usual stability or continuity requirements on the approximation procedure \cite{mhaskar1993approximation,devore1989optimal,maiorov2000near}. Many best-approximation results permit the inner parameters to depend on the target. An important exception, revisited below, is the pre-fabricated construction in \cite{mhaskar1996neural}, in which only the outer coefficients depend on the target. Related estimates for ReLU and ReLU$^k$ activations appear in \cite{liu2024approximation,mao2023rates,mao2024approximation}.

The linearized problem fixes the inner parameters in advance. Given a target-independent set $\displaystyle\Theta_n=\Big\{\binom{w_j}{b_j}\Big\}_{j=1}^n,$ one considers the fixed linear space
\begin{equation}\label{eqn:linearized-network}
    \mathcal L_n(\Theta_n;\sigma)
    :=\operatorname{span}\{\sigma(w_j\cdot\circ+b_j):1\leq j\leq n\},
\end{equation}
and optimizes only the outer coefficients. A single space must therefore approximate the entire target class at the asserted order. For ReLU$^k$, the linearized networks satisfy
\begin{equation}\label{eqn:liu-linearized-rate}
    \inf_{g\in\mathcal L_n}\|f-g\|_{L^2(\Omega)}
    \lesssim
    n^{-\frac rd}
    \|f\|_{H^{r}(\Omega)},\quad r\leq\frac{d+2k+1}{2}
\end{equation}
with fixed inner parameters \cite{liu2025integral}. The homogeneity of ReLU$^k$ is important here: it permits radial scalings of the parameters to be separated from their directions and thereby relates constructions on the sphere to networks on Euclidean domains.

For nonhomogeneous activations, the corresponding picture is less direct. It is proved in \cite{mhaskar1999approximation} that zonal activations with suitable ultraspherical coefficients attain the sharp Sobolev order on the sphere when the directions are quasi-uniform. Thus some analytic activations can support optimal linearized approximation in the spherical setting. On a general domain, however, normalizing an affine parameter changes the activation itself, because the homogeneity used for ReLU$^k$ is no longer available. Consequently, the spherical result does not automatically yield a Euclidean-domain theorem for bounded activations such as $\tanh$.

The ridge-function lifting theorem in \cite{petrushev1998approximation} provides a different route from one dimension to several dimensions. Let $\Omega\subset\RR^d$ be a bounded domain in the class considered there, let $I\subset\RR$ be the corresponding fixed projection interval, and let $X_m\subset L^2(I)$ be an $m$-dimensional space satisfying
\begin{equation}\label{eqn:petrushev-univariate}
    \inf_{v\in X_m}\|u-v\|_{L^2(I)}
    \lesssim m^{-q}\|u\|_{H^q(I)}.
\end{equation}
For an particular family of cubature points $\Omega_m\subset\SS^{d-1}$ satisfying $|\Omega_m|\simeq m^{d-1}$, define
\begin{equation*}
    Y_m:=\operatorname{span}\{v(\omega\cdot\circ):v\in X_m,\ \omega\in\Omega_m\}.
\end{equation*}
While the original lifting argument is formulated using a particular family of cubature points, we show that these points can be replaced by any quasi-uniform family in the Appendix. The required positive cubature weights with polynomial exactness follow from \cite{narcowich2006localized}, while the high-frequency synthesis estimate follows from the separation and localization properties of the quasi-uniform set; related scattered-data quadrature constructions appear in \cite{leGiaMhaskar2009localized}. The resulting quasi-uniform lifting theorem gives $\dim Y_m\lesssim m^d$ and
\begin{equation}\label{eqn:petrushev-multivariate}
    \inf_{g\in Y_m}\|f-g\|_{L^2(\Omega)}
    \lesssim
    m^{-q-\frac{d-1}{2}}
    \|f\|_{H^{q+\frac{d-1}{2}}(\Omega)}.
\end{equation}
Thus, with $r=q+(d-1)/2$ and $n=m^d$, the space $Y_m$ contains at most order $n$ ridge features and \eqref{eqn:petrushev-multivariate} gives the optimal rate $m^{-r}=n^{-r/d}$. Although the direct lifting step requires $r>(d-1)/2$, real interpolation between this smoother endpoint estimate and the trivial $L^2$ estimate extends the same rate to every $r>0$. The multivariate construction is therefore reduced to a sharp one-dimensional Jackson estimate together with a product of arbitrary quasi-uniform directions and prescribed univariate parameters.

The purpose of this paper is to establish the required one-dimensional estimate for a family of analytic activations and, in particular, for $\tanh$. Let $I,I'\subset\RR$ be bounded intervals and let $\{x_j^{(m)}\}_{j=0}^m\subset I'$ be any quasi-Chebyshev family. For the target-independent spaces
\begin{equation*}
    X_m^{\tanh}
    :=\operatorname{span}\{\tanh(x-x_j^{(m)}):0\leq j\leq m\},
\end{equation*}
we prove, for $0\leq s\leq r$,
\begin{equation}\label{eqn:intro-main-result}
    \inf_{g\in X_m^{\tanh}}
    \|f-g\|_{H^s(I)}
    \lesssim
    m^{-(r-s)}\|f\|_{H^r(I)}.
\end{equation}

Taking the one-dimensional instances of these results as the spaces $X_m$ in \eqref{eqn:petrushev-univariate}, and applying the lifting theorem in \cite{petrushev1998approximation} in the quasi-uniform form proved in Appendix~\ref{app:quasiuniform-lifting}, gives the $d$-dimensional sharp rate
\begin{equation}\label{eqn_tan_rate_d}
    \lt\|f-\sum_{\theta\in\Omega_m}\sum_{j=0}^m a_{\theta j}\tanh\lt(\theta\cdot\circ-x_j^{(m)}\rt)\rt\|_{L^2(\Omega)}
    \lesssim m^{-r}\|f\|_{H^r(\Omega)}
    =n^{-\frac rd}\|f\|_{H^r(\Omega)},
    \qquad n=m^d.
\end{equation}
Here the inner parameter set is a product of arbitrary quasi-uniform directions $\Omega_m\subset\SS^{d-1}$, with $|\Omega_m|\simeq m^{d-1}$, and quasi-Chebyshev biases $\{x_j^{(m)}\}_{j=0}^{m}\subset I'$. Hence the number of displayed features is $(m+1)|\Omega_m|\asymp n$. Real interpolation covers the full range $r>0$. A similar argument applies to analytic activations satisfying the Taylor-coefficient condition stated in Theorem~\ref{thm_1d_analytic}, including sigmoid, some rational functions, and $\arctan$. For these activation functions, the parameter set takes a different form from \eqref{eqn_tan_rate_d}, but still relies on a product construction. These rates are sharp in terms of the number of features by the classical width lower bound for Sobolev balls \cite{devore1989optimal}.


A closely related pre-fabricated construction is given in \cite{mhaskar1996neural}. Related Sobolev-norm estimates based on the same finite-difference mechanism appear in \cite{park2026sobolev}. For $m\in\NN$, \cite{mhaskar1996neural} uses approximants of the form
\begin{equation}\label{eqn:mhaskar-approximant}
    f_n(\bx)=\sum_{\bj\in\{-m,\ldots,m\}^d}a_{\bj,n}(f)\,\sigma\bigl(b+h_n\bj\cdot\bx\bigr),
    \qquad
    n=(2m+1)^d\asymp m^d,
\end{equation}
where the inner parameters are prescribed independently of $f$. Under the nonvanishing-derivative assumptions of \cite[Theorem~2.1]{mhaskar1996neural}, a sufficient admissibility condition for approximation in $L^p$, $1\leq p\leq\infty$, is
\begin{equation}\label{eqn:mhaskar-scale}
    h_n\leq\min\left\{\frac{\delta}{3md},
    \min_{0\leq\mathbf{k}\leq2m}
    \left(
    \frac{m^{-r-\alpha}}
    {M_{\sigma;m,d}
    \sum_{0\leq\mathbf{p}\leq\mathbf{k}}
    |\sigma^{(|\mathbf{p}|)}(b)|^{-1}
    |\tau_{\mathbf{k},\mathbf{p}}|}
    \right)^{1/2}
    \right\},
\end{equation}
where $\alpha=d/\min(p,2)$, $M_{\sigma;m,d}$ is the finite-difference remainder constant in that construction, and $\tau_{\mathbf{k},\mathbf{p}}$ are the monomial coefficients of the corresponding tensor-product Chebyshev polynomial. The severity of this condition is already visible in one dimension. Suppose that $\sigma$ extends holomorphically to a neighborhood of $\overline{B_{R+\varepsilon}(b)}$, that $0<\delta<R$, and that $\sigma^{(k)}(b)\neq0$ for every $k\geq0$. Replacing the finite-difference remainder constant by the standard Cauchy majorant gives the explicit strengthened sufficient condition
\begin{equation}\label{eqn:mhaskar-factorial-scale}
    h_n
    \leq
    C m^{-\frac{r+\alpha+1}{2}}
    \frac{\min\{1,(R-\delta)^{m+1}\}}
    {\sqrt{(2m+2)!}}
    =
    \exp\bigl(-m\log m+O(m)\bigr).
\end{equation}
For $d=1$, the network has $n=2m+1$ features, and the diameter of the inner-weight set is at most $2mh_n$. Thus this explicit safe specialization of the finite-difference argument produces a factorially clustered feature family. This concerns the scale supplied by that construction, not a necessary condition for every representation attaining the same approximation order. In contrast, the quasi-Chebyshev parameters used here remain distributed over a fixed interval and have algebraic separation. The approximation exponent and the smooth activation class alone are therefore not new; the distinction pursued here is the prescribed quasi-Chebyshev feature geometry on fixed parameter intervals and its product-type lifting through arbitrary quasi-uniform directions.

The remainder of the paper is organized as follows. Section~\ref{sec:preliminaries} records the geometric and discrete-polynomial tools. Section~\ref{sec:main-results} states the one- and multidimensional approximation results and verifies the analytic non-cancellation condition for several activations. Section~\ref{sec:proofs} contains the proofs, Section~\ref{sec:conclusion} concludes the paper, and Appendix~\ref{app:quasiuniform-lifting} proves the quasi-uniform form of the ridge-function lifting theorem used above.

\section{Preliminaries}\label{sec:preliminaries}
This section fixes the notation and recalls the sampling results used below. We first introduce quasi-uniform and quasi-Chebyshev point families and then record the Marcinkiewicz--Zygmund inequalities needed for the construction of least-square interpolation operator.
 

We use the comparison notation of \cite{xu1992iterative}. The symbols $\gtrsim$, $\lesssim$, and $\simeq$ denote inequalities up to positive constants independent of the approximation parameter unless stated otherwise. Thus, when we write
$$
f(x)\gtrsim g(x),\quad g(x)\lesssim h(x),\quad h(x)\simeq k(x),
$$
there exist positive constants $c_1,c_2,c_3,c_4$, independent of $x$, such that
$$
f(x)\ge c_1 g(x),\quad g(x)\le c_2 h(x),\quad c_3 h(x)\le k(x)\le c_4 h(x).
$$
Throughout, $\NN=\{0,1,2,\ldots\}$ and $\mathbb N_+=\{1,2,\ldots\}$.

\begin{definition}[Quasi-uniform]\label{def_quasi_uniform}
    A set of points $\{x_j^{(m)}\}_{j=1}^m\subset [a,b]$ is said to be quasi-uniform if
    \begin{equation}\label{eqn:quasiuniform-1}
        \mal_{x\in[a,b]}\min\limits_{1\leq j\leq m}|x-x_j^{(m)}|\lesssim\min\limits_{i\neq j}|x_i^{(m)}-x_j^{(m)}|.
    \end{equation}

    A set of points $\{\theta_j^{(N)}\}_{j=1}^N\subset \SS^d$ is said to be quasi-uniform if
    \begin{equation}\label{eqn:quasiuniform-2}
        \mal_{\theta\in \SS^d}\min\limits_{1\leq j\leq N}\rho(\theta,\theta_j^{(N)})\lesssim\min\limits_{i\neq j}\rho(\theta_i^{(N)},\theta_j^{(N)}),
    \end{equation}
    where $\rho(\cdot,\cdot)$ denotes the geodesic distance on $\SS^d$.
    All hidden constants are independent of the cardinality parameter.
\end{definition}
\begin{definition}[Quasi-Chebyshev]\label{def_quasi-chebyshev}
    A set of points $\{x_j^{(m)}\}_{j=1}^m\subset [a,b]$ is said to be quasi-Chebyshev if there exists a set of quasi-uniform points $\{\theta_j^{(m)}\}_{j=1}^m$ on $[0,\pi]$ such that $$x_j^{(m)} = \frac{a+b}{2}+\frac{b-a}{2}\cos\theta_j^{(m)}.$$
\end{definition}

We now prove some elementary lemmas showing that the quasi-Chebyshev structure is stable under transformations that arise in the subsequent analysis.

\begin{lemma}\label{lem:transform-cheby}
    For any quasi-Chebyshev $\{x_j^{(m)}\}_{j=1}^m\subset [a,b]$ and any bi-Lipschitz map $f$, the transformed set $\{f(x_j^{(m)})\}_{j=1}^m\subset [\min (f(a),f(b)), \max (f(a),f(b)) ]$ is also quasi-Chebyshev.
\end{lemma}
    \begin{proof}
        Without loss of generality, assume that $f$ is increasing. By Definition \ref{def_quasi-chebyshev}, it suffices to prove that the following function is bi-Lipschitz on $[0,\pi]$:
        \begin{equation}
            h(\theta) = \arccos\lt( \frac{f(\frac{a+b}{2}+\frac{b-a}{2}\cos\theta) - \frac{f(a) + f(b)}{2}}{\frac{f(b)-f(a)}{2}} \rt).
        \end{equation}
        If we denote $g(x) = \cos x, \ x\in[0,\pi]$ and \begin{equation}
            \begin{split}
                \tilde f\colon [-1,1]&\to [-1,1] \\
                x&\mapsto \frac{f(\frac{a+b}{2}+\frac{b-a}{2}x) - \frac{f(a) + f(b)}{2}}{\frac{f(b)-f(a)}{2}},
            \end{split}
        \end{equation}
        then we only need to prove that $g^{-1}\circ \tilde f\circ g$ is a bi-Lipschitz function on $[0,\pi]$. 

        Since $\tilde f, \tilde f^{-1}$ are Lipschitz, they are absolutely continuous and are differentiable almost everywhere, and there exist $c_f,C_f>0$ such that 
        $$0<c_f\leq \tilde f'(x)\leq C_f,\quad a.e.\ x\in[-1,1].$$
        Then for a.e. $\theta\in (0,\arccos\tilde f^{-1}(0)]$, 
        \begin{equation}
        \begin{split}
            |(g^{-1}\circ \tilde f\circ g)'(\theta)| &= \frac{|\sin\theta|}{\sqrt{1-\tilde f(\cos\theta)^2}}|\tilde f'(\cos\theta)| \\
            &\leq \frac{C_f|\theta|}{\sqrt{1+\tilde f(\cos\theta)}\sqrt{\tilde f(1)-\tilde f(\cos\theta)}} \\
            &\leq \frac{C_f|\theta|}{\sqrt{c_f}\sqrt{1-\cos\theta}} 
            \leq \frac{C_f\sqrt{12}}{\sqrt{c_f}}
        \end{split}
        \end{equation}
        which implies that $(g^{-1}\circ \tilde f\circ g)'(\theta)$ is bounded almost everywhere near $\theta=0$. The same argument applies near $\theta=\pi$. On every compact subinterval of $(0,\pi)$ the composite is absolutely continuous, and the preceding derivative bound is uniform as the subinterval approaches either endpoint. Therefore $g^{-1}\circ \tilde f\circ g$ is Lipschitz on $[0,\pi]$. Since $(g^{-1}\circ \tilde f\circ g)^{-1} = g^{-1}\circ \tilde f^{-1}\circ g$, its inverse is Lipschitz by the same argument, and the proof is complete.
    \end{proof}
    
\begin{lemma}\label{lem:rescale-cheby}
    For any quasi-Chebyshev $\{x_j^{(m)}\}_{j=0}^m\subset [a,b]$, suppose $x_0^{(m)}<x_1^{(m)}<\cdots <x_m^{(m)}$, then the set of points $$ \lt\{y_j^{(m)} = \frac{x_j^{(m)}-x_0^{(m)}}{b-x_0^{(m)}}\rt\}_{j=0}^m\subset [0,1]$$
    is quasi-Chebyshev.
\end{lemma}
\begin{proof}
    Take $f_m(x)=\frac{x-x_0^{(m)}}{b-x_0^{(m)}}$ in Lemma \ref{lem:transform-cheby}. Since $\{x_j^{(m)}\}_{j=0}^m$ is quasi-Chebyshev, $x_0^{(m)}-a\lesssim m^{-2}$. Consequently, the Lipschitz constants of $f_m$ and $f_m^{-1}$ are bounded uniformly in $m$, and Lemma \ref{lem:transform-cheby} first gives a quasi-Chebyshev family on $[f_m(a),1]$. The left endpoint $f_m(a)$ differs from $0$ by $O(m^{-2})$, the scale of the extreme Chebyshev spacing. Passing from $[f_m(a),1]$ to $[0,1]$ therefore preserves the comparability of the angular fill distance and separation, which proves the claim.
\end{proof}

\begin{lemma}\label{lem:limit-cheby}
    Let $\{x_j^{(m)}\}_{j=0}^m\subset [0,1]$ be quasi-Chebyshev and let $\{\Delta_m\}_{m=1}^\infty$ be a positive sequence such that
    $$0<\Delta_m\leq\Delta,\qquad |\Delta_m-\Delta|\lesssim m^{-2},$$
    for some $\Delta>0$. Then the set of points
    $$ \{\Delta_m x_j^{(m)}\}_{j=0}^m \subset [0,\Delta]$$
    is quasi-Chebyshev. The same conclusion holds after removing either left or right endpoint.
\end{lemma}
\begin{proof}
    On the moving interval $[0,\Delta_m]$, the conclusion follows from affine invariance, uniformly in $m$. The right endpoint differs from $\Delta$ by $O(m^{-2})$. Under the cosine parametrization, an $O(m^{-2})$ endpoint displacement produces an $O(m^{-1})$ angular displacement, which is comparable with both the angular fill distance and separation of a quasi-uniform family. The same two quantities therefore remain comparable on the fixed interval $[0,\Delta]$. Removing left or roght endpoint changes the endpoint fill distance by at most one adjacent angular spacing and does not destroy quasi-uniformity.
\end{proof}

\subsection{Properties of the polynomials with scattered points}\label{subsec:prop_scatter}
In this subsection, we derive the Marcinkiewicz-Zygmund inequality for cosine polynomials sampled at quasi-uniform points on an interval, which is used in interpolation error analysis in Lemma \ref{lem:least-square-interp}. These inequalities can be viewed as weighted discrete $L^2$-norm equivalences, and follow from the fact that scattered point sets on $\mathbb S^d$ admit positive quadrature formulas with polynomial exactness; see, for instance, \cite{mhaskar2001spherical,bondarenko2013optimal,jetter2023norming}.
 
Consider a finite subset $\{\theta_j^*\}_{j=1}^N\subset \mathbb{S}^{d}$ comprising distinct, scattered points. The mesh norm for $\{\theta_j^*\}_{j=1}^N\subset \mathbb{S}^{d}$ is defined by

\begin{equation}
h=\mal_{\eta\in\SS^d}\min\limits_{1\leq j\leq N}\rho(\eta,\theta_j^*).
\end{equation}

We have Marcinkiewicz-Zygmund inequality for polynomials on sphere: 
\begin{lemma}[{\cite[Theorem 3.1]{mhaskar2001spherical}}]\label{lem:MZ-inequ}
    Given scattered points $\{\theta_j^*\}_{j=1}^N \subset \SS^d$ with mesh norm $h$, there exist nonnegative weights $\tau_1,\dots,\tau_N$ with $\tau_j\lesssim h^d$ and a constant $C_1$ (independent of $N, h$) such that
    \begin{equation}\label{eqn:MZ-inequ}
        \sum_{j=1}^N\tau_j|p(\theta_j^*)|^2 \simeq \|p\|_{L^2(\SS^d)}^2, \quad \forall p\in\mathbb{P}_{J}(\SS^d),
    \end{equation}
    where $J=\lfloor C_1h^{-1}\rfloor$.
\end{lemma}

By suitably reflecting quasi-uniform points $\{\theta_j\}_{j=1}^m\subset [0,\pi]$ to $\SS^1$, we obtain the following Marcinkiewicz--Zygmund inequality for cosine polynomials.
\begin{lemma}[Marcinkiewicz-Zygmund inequality]\label{lem:MZ-inequ-0-pi}
    Given quasi-uniform points $\{\theta_j\}_{j=1}^m \subset [0,\pi]$, there exist nonnegative weights $\tau_1,\dots,\tau_m$ with $\tau_j\lesssim m^{-1}$ and a constant $C_2$ (independent of $m$) such that
    \begin{equation}\label{eqn:MZ-inequ-0-pi}
        \sum_{j=1}^m\tau_j|p(\theta_j)|^2 \simeq \|p\|_{L^2([0,\pi])}^2, \quad \forall p\in\mathcal{C}_{J},
    \end{equation}
    where $J=\lfloor C_2 m\rfloor$, $\mathcal{C}_{J}={\rm span}\{\cos k\theta:0\leq k\leq J\}$.
\end{lemma}
\begin{proof}
    It suffices to consider $m\geq2$, since the remaining finite case can be absorbed into the constants.
    For any quasi-uniform points $\{\theta_j\}_{j=1}^m \subset [0,\pi]$, suppose $\theta_1<\theta_2<\cdots <\theta_m$, we can extend them to quasi-uniform points $\{\theta_j'\}_{j=1}^{2m-2} \subset \SS^1$ with mesh norm $h \simeq m^{-1}$ in the following way:
    \begin{equation}\label{eqn:point-extension}
        \theta_j' = \left\{
            \begin{aligned}
                &\theta_j, &&\quad j=1,2,\ldots,m, \\
                &-\theta_{j-m+1}, &&\quad j=m+1,m+2,\ldots,2m-2.
            \end{aligned}\right.
        \end{equation}
    Therefore, by Lemma \ref{lem:MZ-inequ} with $d=1$, there exist nonnegative weights $\tilde\tau_1,\dots,\tilde\tau_{2m-2}$ with $\tilde\tau_j\lesssim h$ and a constant $C_1$ (independent of $m$ and $h$) such that
    \begin{equation}\label{eqn:MZ-inequ-periodic}
        \sum_{j=1}^{2m-2}\tilde\tau_j|p(\theta_j')|^2 \simeq \|p\|_{L^2((-\pi,\pi])}^2, \quad \forall p\in\mathcal{T}_{J},
    \end{equation}
    where $J=\lfloor C_1h^{-1}\rfloor$ and $\mathcal{T}_{J}$ is the set of trigonometric polynomials of degree at most $J$. Restricting to the even functions $p\in\mathcal{C}_J$ and defining
    $$\tau_j = \left\{
    \begin{aligned}
        &\tilde\tau_j, &&\quad j=1,m, \\
        &\tilde\tau_j + \tilde\tau_{m-1+j}, &&\quad j=2,3,\ldots,m-1,
    \end{aligned}\right.$$
    equation \eqref{eqn:MZ-inequ-periodic} yields
    \begin{equation}
        \sum_{j=1}^{m}\tau_j|p(\theta_j)|^2 \simeq \|p\|_{L^2((-\pi,\pi])}^2 \simeq \|p\|_{L^2([0,\pi])}^2, \quad \forall p\in\mathcal{C}_{J}.
    \end{equation}
\end{proof}

\section{Main results}\label{sec:main-results}
We first state a direct tensor-grid result for analytic activations satisfying a quantitative non-cancellation condition at the origin. The theorem applies on bounded domains that admit a bounded Sobolev-extension operator and uses a Cartesian product of quasi-Chebyshev weight sets. 

\begin{theorem}\label{thm_1d_analytic}
    Let $\sigma$ be a function analytic in an bounded open complex domain $\Omega_\sigma\subset\CC$. Suppose there exists ${\rho_\sigma}>0$ such that
    \begin{equation}\label{eqn:derivative-bound}
        \frac{1}{k!}\lt|\sigma^{(k)}(0)\rt|\geq C \rho_\sigma^{-k},\quad\forall k\in\NN.
    \end{equation}
    Let $\Omega\subset\RR^d$ be a bounded domain admitting a bounded Sobolev-extension operator, and let $I^d=\prod_{k=1}^d I_k$ be a bounded rectangle such that $\Omega\Subset I^d$. For any $m\in\NN_+$, $f\in H^r(\Omega)$, any bounded rectangle $I'^d=\prod_{k=1}^d I_k'\subset \RR^d$, and any quasi-Chebyshev sets $\{\omega_{k,i}^{(m)}\}_{i=1}^m\subset I'_k$ satisfying 
    \begin{equation}\label{eqn:param-bound}
        {\bm 0}\in I'^d,\qquad \lt\{{\bm\omega}\cdot {\bm x}: \forall{\bm\omega}\in I'^d,{\bm x}\in I^d \rt\}\Subset \Omega_\sigma,
    \end{equation}
    set $n=m^d$. Then there exists $a\in\RR^n$ such that for any $0\leq s\leq r$,
    \begin{equation}\label{eqn:1d-analytic-approx}
        \lt\|f - \sum_{1\leq j_1,\ldots,j_d\leq m} a_\bj \sigma(\bomega_\bj^{(m)}\cdot\circ)\rt\|_{H^s(\Omega)}\lesssim m^{-(r-s)} \|f\|_{H^r(\Omega)} = n^{-\frac{r-s}{d}} \|f\|_{H^r(\Omega)},
    \end{equation}
    where $\bj = (j_1,\ldots,j_d)$, $a_\bj = a_{j_1,\ldots,j_d}$, $\bomega_\bj^{(m)} = \lt(\omega_{1,j_1}^{(m)},\ldots,\omega_{d,j_d}^{(m)}\rt)$.
\end{theorem}

The next theorem treats $\tanh$ with prescribed shifts. Unlike the first result, it does not require a Taylor-coefficient condition at the origin: the rational transformation used in its proof permits an arbitrary quasi-Chebyshev family on the bias interval.

\begin{theorem}\label{thm_main_1d}
Let $I,I'\subset \RR$ be bounded intervals, let $m\in\NN_+$ and $f\in H^r(I)$, and let $\{x_i^{(m)}\}_{i=0}^m\subset I'$ be any quasi-Chebyshev set. Then there exists $a\in\RR^{m+1}$ such that, for every $0\leq s\leq r$,
    \begin{equation}
        \lt\|f - \sum_{i=0}^{m} a_i\tanh(\circ-x_i^{(m)})\rt\|_{H^s(I)}\lesssim m^{-(r-s)} \|f\|_{H^r(I)}.
    \end{equation}
\end{theorem}

The following corollaries lift the one-dimensional cases of the preceding theorems to several dimensions. The construction in \cite{petrushev1998approximation} is used in the form proved in Appendix~\ref{app:quasiuniform-lifting}, where its particular cubature nodes are replaced by an arbitrary quasi-uniform family. In both corollaries, $m$ is the univariate resolution, $n=m^d$ is the total feature scale, and the resulting $L^2$ rate is $m^{-r}=n^{-r/d}$.

\begin{corollary}\label{cor_smooth_act_d}
    Let $d\geq2$, $r>0$, $\Omega\subset\RR^d$ be a bounded domain admitting bounded Sobolev-extension, and the analytic function $\sigma$ be as in Theorem \ref{thm_1d_analytic}. Then for any quasi-uniform set $\Omega_m=\{\theta_i^{(m)}\}_{i=1}^{N_m}\subset \SS^{d-1}$ with $N_m=|\Omega_m|\simeq m^{d-1}$ and any quasi-Chebyshev set $\{\omega_j^{(m)}\}_{j=1}^{m}\subset I'$ satisfying \eqref{eqn:param-bound} on the corresponding projection interval,
    \begin{equation}
        \inf\limits_{\{a_{ij}\}\in\RR^{N_m\times m}}\lt\|f - \sum_{i=1}^{N_m}\sum_{j=1}^m a_{ij}\sigma \lt(\omega_j^{(m)}\theta_i^{(m)}\cdot\circ \rt)\rt\|_{L^2(\Omega)}\lesssim n^{-\frac{r}{d}} \|f\|_{H^r(\Omega)},\quad f\in H^r(\Omega).
    \end{equation}
\end{corollary}

\begin{corollary}\label{cor_tanh_d}
    Let $d\geq2$, $r>0$, and $\Omega\subset\RR^d$ be a bounded domain admitting bounded Sobolev-extension. Then for any quasi-uniform set $\Omega_m=\{\theta_i^{(m)}\}_{i=1}^{N_m}\subset \SS^{d-1}$ with $N_m=|\Omega_m|\simeq m^{d-1}$ and any quasi-Chebyshev set $\{x_j^{(m)}\}_{j=1}^{m}\subset I'$, where $I'\subset \RR$ is a bounded interval,
    \begin{equation}
        \inf\limits_{\{a_{ij}\}\in\RR^{N_m\times m}}\lt\|f - \sum_{i=1}^{N_m}\sum_{j=0}^m a_{ij}\tanh\lt(\theta_i^{(m)}\cdot\circ-x_j^{(m)}\rt)\rt\|_{L^2(\Omega)}\lesssim n^{-\frac{r}{d}} \|f\|_{H^r(\Omega)},\quad f\in H^r(\Omega).
    \end{equation}
\end{corollary}
\begin{proof}[Proof of Corollary \ref{cor_smooth_act_d} and \ref{cor_tanh_d}]
Let $r>0$ be arbitrary. Choose
$$
\rho>\max\left\{r,\frac{d-1}{2}\right\},
\qquad
q=\rho-\frac{d-1}{2}>0.
$$
Let $V_m$ denote the feature space displayed in the corresponding corollary. Applying the one-dimensional case of Theorem~\ref{thm_1d_analytic}, respectively Theorem~\ref{thm_main_1d}, with smoothness $q$, and then applying Theorem~\ref{thm:quasi-uniform-ridge-lifting} gives
\begin{equation}\label{eqn:high-smoothness-endpoint}
\inf_{v\in V_m}\|g-v\|_{L^2(\Omega)}
\lesssim
m^{-\rho}\|g\|_{H^\rho(\Omega)},
\qquad g\in H^\rho(\Omega).
\end{equation}

The real-interpolation identity and the standard embedding of interpolation spaces give
\begin{equation}\label{eqn:interpolation-identity-domain}
H^r(\Omega)
=
\bigl(L^2(\Omega),H^\rho(\Omega)\bigr)_{\vartheta,2}
\hookrightarrow
\bigl(L^2(\Omega),H^\rho(\Omega)\bigr)_{\vartheta,\infty},
\end{equation}
where $\vartheta=\frac{r}{\rho}\in(0,1)$ (see, e.g., \cite[Theorems~3.1.2 and~6.4.5]{bergh2012interpolation}). Consequently, the K-functional satisfies
\begin{equation}\label{eqn:k-functional-estimate}
K(t,f;L^2(\Omega),H^\rho(\Omega))
\lesssim
t^\vartheta\|f\|_{H^r(\Omega)},
\qquad t>0.
\end{equation}
Taking $t=m^{-\rho}$, we may choose $g_m\in H^\rho(\Omega)$ such that
\begin{equation}\label{eqn:smooth-intermediate-approximation}
\|f-g_m\|_{L^2(\Omega)}
+
m^{-\rho}\|g_m\|_{H^\rho(\Omega)}
\lesssim
m^{-\rho\vartheta}\|f\|_{H^r(\Omega)}
=
m^{-r}\|f\|_{H^r(\Omega)}.
\end{equation}

By \eqref{eqn:high-smoothness-endpoint}, there exists $v_m\in V_m$ such that
$$
\|g_m-v_m\|_{L^2(\Omega)}
\lesssim
m^{-\rho}\|g_m\|_{H^\rho(\Omega)}.
$$
Therefore,
\begin{align*}
\|f-v_m\|_{L^2(\Omega)}
&\leq
\|f-g_m\|_{L^2(\Omega)}
+
\|g_m-v_m\|_{L^2(\Omega)} \\
&\lesssim
\|f-g_m\|_{L^2(\Omega)}
+
m^{-\rho}\|g_m\|_{H^\rho(\Omega)} \lesssim
m^{-r}\|f\|_{H^r(\Omega)}.
\end{align*}
Since $v_m$ belongs to the feature space displayed in the corresponding corollary, it has the asserted linearized neural-network representation. Moreover, $V_m$ contains at most order $n=m^d$ prescribed features, and hence
$$
m^{-r}=n^{-r/d}.
$$
\end{proof}

The next theorem gives a generic mechanism for the non-cancellation condition \eqref{eqn:derivative-bound}; Theorem \ref{thm_assumption_verify} then records concrete activations for which the condition can be verified.

\begin{theorem}\label{thm_assumption_general}
Let $\sigma$ be real analytic on the real axis and admit the meromorphic continuation described below. Define
    $$
    \sigma_t(x)=\sigma(x+t).
    $$
    Suppose that $I$ is a finite interval and that, for every $t\in I$, the same conjugate pair of simple poles $p_\ast,\overline{p_\ast}$ of $\sigma$ gives the unique nearest singularities $p_\ast-t,\overline{p_\ast}-t$ of $\sigma_t$, with a common annular neighborhood free of other singularities. Then, for almost every $t\in I$, $\sigma_t$ satisfies assumption \eqref{eqn:derivative-bound} in Theorem \ref{thm_1d_analytic}; that is, there exist positive constants $C$ and ${\rho_t}$, possibly depending on $t$, such that
    \begin{equation}
        \frac{1}{k!}\lt|\sigma_t^{(k)}(0)\rt|\geq C \rho_t^{-k},\quad\forall k\in\NN.
    \end{equation}
\end{theorem}
\begin{theorem}\label{thm_assumption_verify}
    The following activation functions satisfy the assumption \eqref{eqn:derivative-bound}.
    \begin{enumerate}
    \item $\displaystyle\sigma(x) = \tanh(x+\frac{q\pi}{2})$ for $\displaystyle q\in\QQ\backslash\{0,\pm 1\}$.
    \item $\displaystyle\sigma(x) = \frac{1}{1+e^{-x+q\pi }}$ for $q\in\QQ\backslash\{0,\pm 1\}$.
    \item $\displaystyle\sigma(x) = \frac{c_1 + c_2 x}{1+x^2}$ for $c_1,c_2\neq 0$.
    \item $\displaystyle\sigma(x) = \frac{1}{1+(x+q)^2}$ for $q\in\QQ\backslash\{0,\pm 1\}$.
    \item $\displaystyle\sigma(x) = \arctan(x+q)$ for $q\in\QQ\backslash\{0,\pm 1\}$.
    \end{enumerate}
\end{theorem}

\section{Proofs of the main results}\label{sec:proofs}
We first introduce a least-squares interpolation operator and establish its approximation property. Interpolation in the parameter variables then yields linearized-network approximants to polynomials with exponentially small errors, from which Theorem \ref{thm_1d_analytic} follows. Finally, a change of variables and the rational structure of $\tanh$ reduce Theorem \ref{thm_main_1d} to the analytic construction.
\subsection{Least-squares interpolation for analytic functions}
For any bounded interval $I\subset\RR$, we define a least-square interpolation operator $\mathcal I_m\colon C(I)\to \mathbb P_J(I)$ associated with points $\{x_j\}_{j=1}^m\subset I$:
\begin{equation}\label{eqn:interp-def}
    \mathcal I_m f = \arg\min_{p_J\in \mathbb P_J} \sum_{j=1}^m \tau_j |f(x_j)-p_J(x_j)|^2,
\end{equation}
where $\{\tau_i\}_{i=1}^m$ and $J$ are from Lemma \ref{lem:MZ-inequ-0-pi}.
\begin{lemma}\label{lem:least-square-interp}
    Consider a quasi-Chebyshev set $\{x_j^{(m)}\}\subset[-1,1]$ and its corresponding least-squares interpolation operator $\mathcal I_m$. Then $\mathcal I_m f$ is uniquely determined, $\mathcal I_m$ is linear, and it can be written in the following form: 
\begin{equation}\label{eqn_interp_oper}
    \mathcal I_m f(x) = \sum_{j=1}^m f(x_j^{(m)})\,l_j^{(m)}(x), 
\end{equation}
    where $l_j^{(m)}\in \mathbb P_J([-1,1])$ are only determined by $\{x_j^{(m)}\}$.
    
    Moreover, suppose $f\colon [-1,1]\to\RR$ is analytic inside and on the Bernstein ellipse $\mathcal{ E}_\rho :=\{\frac{1}{2}(\rho e^{i\theta} + \rho^{-1} e^{-i\theta}):\theta\in [0,2\pi)\}$ for some $\rho>1$, then for any $1<\rho_1<\rho$, 
    $$\sup_{z\in \overline{\mathcal{D}_{\rho_1}}} |f(z)-\mathcal I_m f(z)|\leq C_{\rho,\rho_1}M_{f,\rho}  \lt(\frac{\rho_1}{\rho}\rt)^J \sqrt{J},$$
    where $\mathcal{D}_{\rho_1}$ is the bounded domain enclosed by $\mathcal{E}_{\rho_1}$, $M_{f,\rho}:=\sup_{z\in\mathcal{E}_\rho} |f(z)|$.
\end{lemma}
\begin{proof}
    Under the transformation $x = \cos \theta$, it is equivalent to consider the least-squares interpolation operator on $\theta\in[-\pi,\pi]$: 
    \begin{equation}
        \mathcal I_m^\theta F(\theta) := \arg\min_{p_J\in \mathcal{C}_J} \sum_{j=1}^m \tau_j |F(\theta_j) - p_J(\theta_j)|^2,
    \end{equation}
    where $\mathcal I_m^\theta$ is defined for continuous and periodic even functions on $[-\pi,\pi]$, 
    $$\mathcal{C}_J:={\rm span}\{\cos(k\theta):0\leq k\leq J\},$$ 
    $\{\theta_j\}_{j=1}^m = \{\arccos x_j\}_{j=1}^m \subset [0,\pi]$ is quasi-uniform by Definition \ref{def_quasi-chebyshev}. Since $\cos(k\theta)$ can be written as a polynomial of $\cos\theta$ with order $k$ and $\cos^k\theta$ can be written as linear combination of functions in $\mathcal{C}_k$, we have
    $$ (\mathcal I_m f)(\cos\theta) = (\mathcal I_m^\theta f(\cos\circ))(\theta).$$

    We claim that: $\mathcal I_m^\theta F$ is uniquely determined for any continuous and periodic even function $F$ on $[-\pi,\pi]$, and 
    $\mathcal I_m^\theta$ is a linear operator can be written in the following form: 
    \begin{equation}
        \mathcal I_m^\theta F(\theta) = \sum_{j=1}^m F(\theta_j)\,l_j^{(m)}(\cos\theta),
    \end{equation}
    which will prove \eqref{eqn_interp_oper}.

    In fact, suppose that \(p_J,q_J\in \mathcal C_J\) are two minimizers and set
    \[
        r_J:=\frac{p_J+q_J}{2}.
    \]
    Let
    \[
    m_F:=\sum_{j=1}^m \tau_j |F(\theta_j)-p_J(\theta_j)|^2
     =\sum_{j=1}^m \tau_j |F(\theta_j)-q_J(\theta_j)|^2
    \]
    be the minimal value. By the minimality of \(p_J\) and \(q_J\), we have
    \[
    m_F\le \sum_{j=1}^m \tau_j |F(\theta_j)-r_J(\theta_j)|^2 .
    \]
    On the other hand, using the identity
    \[
    \left|\frac{a+b}{2}\right|^2
    =
    \frac{|a|^2+|b|^2}{2}
    -\frac{|a-b|^2}{4},
    \]
    with
    \[
    a=F(\theta_j)-p_J(\theta_j),
    \qquad
    b=F(\theta_j)-q_J(\theta_j),
    \]
    we obtain
    \[
    \begin{aligned}
    \sum_{j=1}^m \tau_j |F(\theta_j)-r_J(\theta_j)|^2
    &=
    \frac12 \sum_{j=1}^m \tau_j |F(\theta_j)-p_J(\theta_j)|^2
    +\frac12 \sum_{j=1}^m \tau_j |F(\theta_j)-q_J(\theta_j)|^2 \\
    &\quad
    -\frac14 \sum_{j=1}^m \tau_j |p_J(\theta_j)-q_J(\theta_j)|^2 \\
    &=
    m_F-\frac14 \sum_{j=1}^m \tau_j |p_J(\theta_j)-q_J(\theta_j)|^2 .
    \end{aligned}
    \]
    Combining this with the minimality inequality gives
    \[
    \sum_{j=1}^m \tau_j |p_J(\theta_j)-q_J(\theta_j)|^2=0.
    \]
    Since each term in the sum is nonnegative, it follows that
    \[
    \tau_j |p_J(\theta_j)-q_J(\theta_j)|^2=0,
    \qquad j=1,\ldots,m.
    \]
    Combining with Lemma \ref{lem:MZ-inequ-0-pi} yields $\|p_J-q_J\|_{L^2((0,\pi])}=0$. 

    Since the minimizer is unique, the least-squares operator $\mathcal{I}_m^\theta$ is well defined. Its normal equations are linear in the nodal vector
    \(
    (F(\theta_1),\ldots,F(\theta_m)).
    \)
    The sampling Gram matrix is invertible by Lemma \ref{lem:MZ-inequ-0-pi}; hence the unique minimizer depends linearly on that vector.
    Equivalently, there exist functions \(l_j^{(m)}(\cos\theta)\in \mathcal C_J\) such that
    \[
    \mathcal I_m^\theta F(\theta)
    =
    \sum_{j=1}^m F(\theta_j)\,l_j^{(m)}(\cos\theta).
    \]
    Indeed, \(l_j^{(m)}(\cos\theta)\) can be chosen as the image, under the same least-squares operator, of the nodal data vector \(e_j\), whose \(j\)-th component is \(1\) and whose other components are \(0\). 

    For the second part of the lemma, denote
    $$\mathcal{S}_\rho = \{w\mid {\rm Re}\,w\in [-\pi,\pi], {\rm Im}\,w\in[0,\log\rho]\},\quad \rho\geq 1.$$
    It suffices to prove the following estimate for a periodic even analytic function $F$:

    \begin{equation}
        \sup_{w\in \mathcal{S}_{\rho_1}} |F(w)-\mathcal I_m^\theta F(w)|\leq C_{\rho,\rho_1}M_{F,\rho}  \lt(\frac{\rho_1}{\rho}\rt)^J \sqrt J,
    \end{equation}

    First, it is known that (see, e.g., \cite[Chapter 7.8]{devore1993constructive}) the Fourier coefficients of
    $$ F(w) = \sum_{k=0}^{+\infty} \hat F(k)\cos kw,$$  
    satisfies
    $$ |\hat{F}(k)|\lesssim M_{F,\rho}\rho^{-k}.$$
    Take \begin{equation}\label{eqn:optimal-pj}
        p_J(w) = \sum_{k=0}^J \hat F(k)\cos kw,
    \end{equation}
    we obtain \begin{equation}\label{eqn:optimal-approx}
    \begin{split}
        \sup_{w\in \mathcal{S}_{\rho_1}} |F(w)-p_J(w)| \leq \sup_{w\in \mathcal{S}_{\rho_1}} \sum_{k=J+1}^{+\infty} |\hat F(k)||\cos kw|\lesssim M_{F,\rho}\sum_{k=J+1}^{+\infty} \rho^{-k}\rho_1^k \lesssim M_{F,\rho}\lt( \frac{\rho_1}{\rho} \rt)^J
    \end{split} 
    \end{equation} 

    
    From $\mathcal{I}_m^\theta F\in \mathcal{C}_J$ and the standard Bernstein inequality for trigonometric polynomials (see, e.g. \cite[Chapter 4]{devore1993constructive}), we have
    $$
        \|\mI_m^\theta F \|_{L^\infty([-\pi,\pi])}\lesssim J^{1/2} \|\mI_m^\theta F \|_{L^2([-\pi,\pi])},
    $$
    by Lemma \ref{lem:MZ-inequ-0-pi},
    $$\|\mI_m^\theta F \|_{L^2([-\pi,\pi])}\simeq \lt( \sum_{j=1}^m \tau_j |\mI_m^\theta F(\theta_j)|^2 \rt)^{1/2},$$
    The sampled values of $\mathcal I_m^\theta F$ are the orthogonal projection of the nodal vector of $F$ onto the sampled cosine-polynomial space. Consequently,
    $$\lt( \sum_{j=1}^m \tau_j |\mI_m^\theta F(\theta_j)|^2 \rt)^{1/2}\leq \lt( \sum_{j=1}^m \tau_j |F(\theta_j)|^2 \rt)^{1/2}\lesssim \|F\|_{L^\infty([-\pi,\pi])}.$$
    Combining inequalities above we obtain
    \begin{equation}\label{eqn:stab-of-interp}
        \|\mI_m^\theta F \|_{L^\infty([-\pi,\pi])}\lesssim J^{1/2} \|F \|_{L^\infty([-\pi,\pi])}.
    \end{equation}

    Now using $\mI_m^\theta p_J = p_J$ we obtain
    \begin{equation}\label{eqn:interp-error}
        \|F-\mI_m^\theta F\|_{L^\infty(\mathcal{S}_{\rho_1})} = \|F-p_J - \mI_m^\theta (F-p_J)\|_{L^\infty(\mathcal{S}_{\rho_1})} \leq \|F-p_J\|_{L^\infty(\mathcal{S}_{\rho_1})} + \|\mI_m^\theta (F-p_J)\|_{L^\infty(\mathcal{S}_{\rho_1})}.
    \end{equation}
    For the second term, note that, for any $q_J\in\mathcal{C}_J$, $e^{iJw}q_J(w)$ is a polynomial of degree at most $2J$ in $z=e^{iw}$. The map $z=e^{iw}$ sends $\mathcal{S}_{\rho_1}$ into the unit disc $\{|z|\leq 1\}$, and the resulting polynomial is analytic on that disc. Therefore, by the maximum principle,
    $$\sup_{w\in\mathcal{S}_{\rho_1}} |e^{iJw}q_J(w)|\leq \sup_{w\in[-\pi,\pi]}|q_J(w)|,$$
    and 
    $$\sup_{w\in\mathcal{S}_{\rho_1}} |q_J(w)|\leq \rho_1^J\sup_{w\in[-\pi,\pi]}|q_J(w)|.$$
    By taking $q_J = \mI_m^\theta(F-p_J)$, we get
    \begin{equation}
        \|\mI_m^\theta(F-p_J)\|_{L^\infty(\mathcal{S}_{\rho_1})}\leq \rho_1^J \|\mI_m^\theta(F-p_J)\|_{L^\infty([-\pi,\pi])}.
    \end{equation}
    Together with \eqref{eqn:interp-error}, \eqref{eqn:stab-of-interp} and \eqref{eqn:optimal-approx}, we get
    \begin{equation}
        \|F-\mI_m^\theta F\|_{L^\infty(\mathcal{S}_{\rho_1})} \lesssim \|F-p_J\|_{L^\infty(\mathcal{S}_{\rho_1})} + \rho_1^J\,J^{1/2}\|F-p_J\|_{L^\infty([-\pi,\pi])} \lesssim M_{F,\rho} \lt(\frac{\rho_1}{\rho}\rt)^J \sqrt J.
    \end{equation}
\end{proof}

The preceding interpolation error estimate is stated on the reference interval \([-1,1]\). By an affine change of variables, the same estimate can be transferred to any bounded interval. More precisely, we obtain the following corollary.

\begin{corollary}\label{cor:least-square-interp}
    For any bounded interval $I\subset \RR$, consider a quasi-Chebyshev set $\{x_j^{(m)}\}\subset I$ and its corresponding least-squares interpolation operator $\mathcal I_m$. Suppose $f\colon I\to\RR$ is analytic inside and on the Bernstein ellipse 
    $$\mathcal{E}_{I,\rho} :=\lt\{x_I + \frac{|I|}{4}(\rho e^{i\theta} + \rho^{-1} e^{-i\theta}):\theta\in [0,2\pi)\rt\},$$ where $\rho>1$, $x_I$ is the mid-point of interval $I$. Then for any $1<\rho_1<\rho$, 
    $$\sup_{z\in \mathcal{D}_{I,\rho_1}} |f(z)-\mathcal I_m f(z)|\leq C_{\rho,\rho_1}M_{f,\rho}  \lt(\frac{\rho_1}{\rho}\rt)^J \sqrt{J},$$
    where $\mathcal{D}_{I,\rho_1}$ is the bounded domain enclosed by $\mathcal{E}_{I,\rho_1}$, $M_{f,\rho}:=\sup_{z\in\mathcal{E}_{I,\rho}} |f(z)|$.
\end{corollary}

The least-squares interpolation estimate extends to rectangles by tensorization. Let $I^d:= I_1\times\cdots\times I_d$, where every $I_k\subset\RR$ is a bounded interval, and consider quasi-Chebyshev sets $\{x_{k,j}^{(m)}\}_{j=1}^m \subset I_k$, $k=1,2,\ldots,d$. For a function \(f\) of \(d\) variables, define the interpolation operator associated with $\{x_{k,j}^{(m)}\}_{j=1}^m$ and acting only in the \(k\)-th variable by
\[
\bigl(\mathcal{I}_m^{(k)}f\bigr)(z_1,\dots,z_d)
:=
\sum_{j=1}^{m}
f(z_1,\dots,z_{k-1},{x_{k,j}^{(m)}},z_{k+1},\dots,z_d)\,l_j(z_k).
\]
Here and below, the dependence of the one-dimensional basis function $l_j$ on the coordinate $k$ and the level $m$ is suppressed.
The full tensor-product interpolation operator is defined by
\begin{equation}
    \mathcal{I}_m^{\otimes d}
:=
\mathcal{I}_m^{(1)}\mathcal{I}_m^{(2)}\cdots\mathcal{I}_m^{(d)}.
\end{equation}

Since the operators act on different variables, they commute, and the full operator has the following representation
\begin{equation}\label{eqn:multi-dim-interp}
    \lt(\mathcal{I}_m^{\otimes d}f\rt)(z_1,\dots,z_d) = \sum_{1\leq j_1,\ldots,j_d\leq m} 
    f\lt(x_{1,j_1}^{(m)},\ldots,x_{d,j_d}^{(m)}\rt) \prod_{k=1}^d l_{j_k}(z_k).
\end{equation}
To write the above expression more compactly, denote 
$${\bm z} = (z_1,\ldots,z_d),\quad \bx_\bj^{(m)} = \lt(x_{1,j_1}^{(m)},\ldots,x_{d,j_d}^{(m)}\rt),\quad l_\bj({\bm z}) = \prod_{k=1}^d l_{j_k}(z_k),$$ \eqref{eqn:multi-dim-interp} can be written equivalently as
\begin{equation}
    \lt(\mathcal{I}_m^{\otimes d}f\rt)({\bm z}) = \sum_{1\leq j_1,\ldots,j_d\leq m} 
    f(\bx_\bj^{(m)})\, l_\bj({\bm z}).
\end{equation}

\begin{corollary}
\label{cor:tensor-interp}
For any rectangle $I^d = I_1\times\cdots\times I_d \subset \RR^d$, consider quasi-Chebyshev sets $\{x_{k,j}^{(m)}\}_{j=1}^m\subset I_k$, $k=1,2,\ldots,d$, and its corresponding interpolation operator $\mathcal I_m^{\otimes d}$. Suppose that
\(f:I^d\to\RR\) is analytic inside and on the Bernstein ellipse
$$ \mathcal{E}_{I^d,\rho} = \mathcal{E}_{I_1,\rho} \times \cdots \times \mathcal{E}_{I_d,\rho},$$
where $\rho>1$. Then for any $1<\rho_1<\rho$, 
there exists a constant \(C_{d,\rho,\rho_1}>0\), independent of
\(m\) and \(f\), such that
\[
\sup_{\bm z\in\overline{\mathcal{D}_{I^d,\rho_1}}}
\left|
f(\bm z)-\mathcal{I}_m^{\otimes d}f(\bm z)
\right|
\le
C_{d,\rho,\rho_1}\,
M_{f,\rho}
\left(\frac{\rho_1}{\rho}\right)^J\sqrt{J},
\]
where $\mathcal{D}_{I^d,\rho_1}$ is the domain enclosed by $\mathcal{E}_{I^d,\rho_1}$, and
$
M_{f,\rho}
:=
\sup_{\bm z\in\mathcal{E}_{I^d,\rho}}|f(\bm z)|.
$
\end{corollary}

\begin{proof}

For \(\ell=0,\dots,d\), define the mixed polyellipses
\[
K_\ell
:=
\prod_{k=1}^\ell \overline{\mathcal{D}_{I_k,\rho_1}}
\times
\prod_{k=\ell+1}^d \overline{\mathcal{D}_{I_k,\rho}}.
\]

Also define the partially interpolated functions
\[
F_0:=f,
\qquad
F_\ell
:=
\mathcal{I}_m^{(\ell)}\mathcal{I}_m^{(\ell-1)}\cdots\mathcal{I}_m^{(1)}f,
\qquad \ell=1,\dots,d.
\]
In particular,
\[
F_d=\mathcal{I}_m^{\otimes d}f.
\]

We first prove, by induction on \(\ell\), that
\begin{equation}
\label{eq:stability}
\sup_{{\bm z}\in K_\ell} |F_\ell|
\leq
(1+\eta_m)^\ell M_{f,\rho},
\qquad \ell=0,\dots,d,
\end{equation}
where $$\eta_m = C_{\rho,\rho_1}\lt(\frac{\rho_1}{\rho}\rt)^J\sqrt{J}$$ 
given in the right hand side of Lemma \ref{lem:least-square-interp} and Corollary \ref{cor:least-square-interp}.
For \(\ell=0\), this is exactly
\[
\sup_{{\bm z}\in K_0} |F_0|
=
\sup_{{\bm z}\in \overline{\mathcal{D}_{I^d,\rho}}} |f|
=
M_{f,\rho}.
\]

Assume that \eqref{eq:stability} holds for \(\ell-1\). Regard all variables
except \(z_\ell\) as passive variables. Since \(F_{\ell-1}\) is holomorphic
in \(z_\ell\) on \(\overline{\mathcal{D}_{I_\ell,\rho}}\), the interpolation error
estimate in Corollary~\ref{cor:least-square-interp} gives
\[
\sup_{{\bm z}\in K_\ell} |F_\ell|
=
\sup_{{\bm z}\in K_\ell} \lt|\mathcal{I}_m^{(\ell)}F_{\ell-1}\rt|
\le
(1+\eta_m) \sup_{{\bm z}\in K_{\ell-1}} |F_{\ell-1}|.
\]
Using the induction hypothesis,
\eqref{eq:stability} is proved for every \(\ell\).

Next, the interpolation estimate in Corollary~\ref{cor:least-square-interp} again applied in
the \(\ell\)-th variable, yields
\begin{equation}
\begin{split}
\sup_{{\bm z}\in K_\ell} |F_{\ell-1}-F_\ell| 
&=
\sup_{{\bm z}\in K_\ell} \lt|F_{\ell-1}-\mathcal{I}_m^{(\ell)}F_{\ell-1}\rt|
\\
&\le
\eta_m \sup_{{\bm z}\in K_{\ell-1}} |F_{\ell-1}| 
\\
&\le
\eta_m(1+\eta_m)^{\ell-1} M_{f,\rho}.
\end{split}
\end{equation}

By telescoping and noticing $K_d\subset K_\ell$ for any $1\leq \ell\leq d$, we obtain
\begin{equation}
\begin{split}
\sup_{{\bm z}\in K_d} \lt|f-\mathcal{I}_m^{\otimes d}f \rt|
&\le
\sum_{\ell=1}^{d}
\sup_{{\bm z}\in K_d} |F_{\ell-1}-F_\ell| \\
&\le
M_{f,\rho}
\sum_{\ell=1}^{d}
\eta_m(1+\eta_m)^{\ell-1} \\
&=
M_{f,\rho}\lt( (1+\eta_m)^d-1 \rt)  \\ 
&\leq C_{d,\rho,\rho_1} M_{f,\rho} \eta_m = C_{d,\rho,\rho_1} M_{f,\rho} \lt(\frac{\rho_1}{\rho}\rt)^J \sqrt{J}.
\end{split}
\end{equation}
\end{proof}

\begin{lemma}\label{lem:legendre-coeff-bound}
    For Legendre polynomials $\{p_i\}_{i=0}^\infty$ on $[-1,1]$:
    $$\int_{-1}^1 p_i(x)\,p_j(x)\,dx = \delta_{ij},$$
    suppose $$p_k(x) = \sum_{j=0}^k p_{k,j} \,x^j,$$ 
    then there exists $\Lambda >4$ such that
    $$|p_{k,j}|\lesssim \Lambda^k,\quad\forall \,0\leq j\leq k,$$
    where the hidden constant does not depend on $k,j$.
\end{lemma}
\begin{proof}
    From Rodrigues formula \cite[Chapter 4.3]{szego1975orthogonal}, one can obtain that
    $$ p_{k,k-2j} = 2^{-k}\sqrt{\frac{2k+1}{2}}(-1)^j \binom{k}{j}\binom{2k-2j}{k}. $$
    By the simple estimate $\binom{k}{j}\leq 2^k$ we obtain
    $$ |p_{k,k-2j}| \leq k 2^{-k}2^k 2^{2k}\lesssim \Lambda^k, $$
    by choosing $\Lambda>4$.
\end{proof}

By affine transformation, the coefficients of Legendre polynomials on an arbitrary interval satisfy a similar estimate.

\begin{corollary}\label{cor:legendre-coeff-bound}
    For Legendre polynomials $\{p_i^I\}_{i=0}^\infty$ on interval $I$:
    $$\int_I p_i^I(x)\,p_j^I(x)\,dx = \delta_{ij},$$
    suppose $$p_k^I(x) = \sum_{j=0}^k p^I_{k,j} \,x^j,$$ 
    then there exists $\Lambda_I>0 $ such that
    $$|p^I_{k,j}|\lesssim \Lambda_I^k,\quad\forall \,0\leq j\leq k,$$
    where the hidden constant does not depend on $k,j$.
\end{corollary}

The corresponding coefficient estimate on an arbitrary rectangle in $\RR^d$ is an immediate consequence of Corollary \ref{cor:legendre-coeff-bound}. 

\begin{corollary}\label{cor:rect-legendre-coeff-bound}
    For any rectangle $I^d = \prod_{k=1}^d I_k \subset \RR^d$, define
    \begin{equation}
        p_\balpha(\bx) = \prod_{k=1}^d p_{\alpha_k}^{I_k}(x_k).
    \end{equation}
    Then $\{p_\balpha(\bx)\}_{0\leq \alpha_1,\ldots,\alpha_d<\infty}$ can be chosen as the Legendre orthogonal polynomials on $I^d$. Suppose 
    $$ p_\balpha(\bx) = \sum_{0\leq \bbeta \leq \balpha} p_{\balpha,\bbeta}\,\bx^\bbeta, $$
    where $0\leq \bbeta \leq \balpha$ means $0\leq \beta_k\leq \alpha_k$ for any $k=1,2,\ldots,d$, and $\bx^\bbeta:=x_1^{\beta_1}\cdots x_d^{\beta_d}$.
    Then there exists $\Lambda_{I^d}$ such that 
    $$|p_{\balpha,\bbeta}|\lesssim \Lambda_{I^d}^{|\balpha|},\quad\forall \,0\leq \bbeta\leq\balpha,$$
    where $|\balpha|=\sum_{k=1}^d \alpha_k$, and the hidden constant does not depend on $\balpha,\bbeta$.
\end{corollary}
\begin{proof}
    Noticing that $$ p_{\balpha,\bbeta} = \prod_{k=1}^d p_{\alpha_k,\beta_k}^{I_k},$$
    take $\Lambda_{I^d} = \max_{1\leq k\leq d} \Lambda_{I_k}$, Corollary \ref{cor:legendre-coeff-bound} then yields
    $$ |p_{\balpha,\bbeta}| \lesssim \prod_{k=1}^d \Lambda_{I_k}^{\alpha_k} \leq \Lambda_{I^d}^{|\balpha|}.$$
\end{proof}

\subsection{Proof of Theorem \ref{thm_1d_analytic}}
\begin{proof}
    Let $E_\Omega$ be a bounded extension operator for $\Omega$. Replacing $f$ by the restriction of $E_\Omega f$ to $I^d$, and retaining the notation $f$, we have
    $$
        \|f\|_{H^r(I^d)}\lesssim \|f\|_{H^r(\Omega)}.
    $$
    It therefore suffices to construct the approximation on $I^d$ and restrict it to $\Omega$ at the end.
    Recall that $n=m^d$. In what follows, we denote 
    $$V_n = {\rm span}\lt\{ \sigma(\bomega_\bj^{(m)}\cdot \bx):1\leq j_1,\ldots,j_d\leq m\rt\},$$ 
    and for any $\balpha = (\alpha_1,\ldots,\alpha_d)\in\NN^d$, denote
    $$
    \balpha! = \alpha_1!\cdots\alpha_d!, \quad
    \partial_\bomega^\balpha = \partial_{\omega_1}^{\alpha_1}\cdots \partial_{\omega_d}^{\alpha_d},\quad 
    \bx^\balpha = x_1^{\alpha_1}\cdots x_d^{\alpha_d}$$

    \textbf{Step 1. Approximation to polynomials.} 
    Fix $\bx\in I^d$, Consider the analytic function 
    $$\sigma_\bx(\bomega):=\sigma(\bomega\cdot\bx), \quad\bomega\in I'^d,$$
    whose derivative with respect to $\bomega$ at $\bomega=0$ can be calculated as
    \begin{equation}\label{eqn:derivative-of-sigmax}
        \partial_\bomega^\balpha \sigma_\bx(0) =  \sigma^{(|\balpha|)}(0)\bx^\balpha,\quad\forall \balpha\in\NN^d,
    \end{equation}
    where $\balpha = (\alpha_1,\ldots,\alpha_d)\in\NN^d$. 
    
    To approximate 
        $\bx^\balpha = \frac{1}{\sigma^{(|\balpha|)}(0)}\partial_\bomega ^\balpha \sigma_\bx(0)$
    using elements in $V_n$, consider the least-squares interpolant $\mathcal I_m ^{\otimes d} \sigma_\bx(\bomega)$ at the tensorized quasi-Chebyshev points $\{\bomega_\bj^{(m)}\}\subset I'^d$: 
    $$\mathcal I_m^{\otimes d} \sigma_\bx(\bomega) = \sum_{1\leq j_1,\ldots,j_d\leq m} \sigma_\bx(\bomega_\bj^{(m)})\,l_\bj^{(m)}(\bomega),$$
    
    Since $\sigma$ is analytic in $\Omega_\sigma$, and recall \eqref{eqn:param-bound}   
    $$
        \lt\{{\bm\omega}\cdot {\bm x}: \forall{\bm\omega}\in I'^d,{\bm x}\in I^d \rt\}\Subset \Omega_\sigma,
    $$
    by compactness, there exists $\rho>1$ such that $\sigma_\bx(\bomega)$ is analytic inside and on the polyellipse $\mathcal{E}_{I'^d,\rho}$ and
    $$
        K_{\sigma,\rho}:=
        \lt\{{\bm\omega}\cdot {\bm x}: {\bm\omega}\in \overline{\mathcal D_{I'^d,\rho}},\ {\bm x}\in I^d \rt\}\Subset \Omega_\sigma.
    $$ 
    
    
    Since $K_{\sigma,\rho}$ is compactly contained in $\Omega_\sigma$,
    \begin{equation}
            M_{\bx,\rho}:= \sup_{\bomega\in\overline{\mathcal{D}_{I'^d,\rho}}} |\sigma_\bx(\bomega)|
            \leq \max_{z\in K_{\sigma,\rho}}|\sigma(z)|=:M_\sigma,
            \quad\forall \bx\in I^d.
    \end{equation}
    Applying Corollary \ref{cor:tensor-interp} to $\sigma_\bx$, and fixing $0<\theta<1$ and $1<\rho_1=\theta\rho<\rho$, gives
    \begin{equation}
        \sup_{\bomega\in \overline{\mathcal{D}_{I'^d,\rho_1}}} |\sigma_\bx(\bomega)-\mathcal I_m^{\otimes d} \sigma_\bx(\bomega)|\leq C_{d,\rho,\rho_1} M_{\sigma}\,\theta^J \sqrt{J},\quad\forall \bx\in I^d.
    \end{equation}

    Similarly, for any $\bbeta\in\NN^d$, the corresponding derivatives are uniformly bounded in $\bx$:
    \begin{equation}
        \sup_{\bomega\in\overline{\mathcal{D}_{I'^d,\rho}}} |\partial_\bx^\bbeta \sigma_\bx(\bomega)|
        = \sup_{\bomega\in\overline{\mathcal{D}_{I'^d,\rho}}} |\bomega^\bbeta \sigma^{(|\bbeta|)}(\bomega\cdot\bx)|
        \leq C_{\sigma,\bbeta},\quad\forall \bx\in I^d,
    \end{equation}
    where $C_{\sigma,\bbeta}$ is finite by compactness. Noticing that 
    $$
    \partial_\bx^\bbeta \lt[\mathcal I_m^{\otimes d} \sigma_\bx \rt](\bomega) =\sum_{1\leq j_1,\ldots,j_d\leq m} \partial_\bx^\bbeta \sigma_\bx(\bomega_\bj^{(m)})\,l_\bj^{(m)}(\bomega)= \mathcal I_m^{\otimes d}(\partial_\bx^\bbeta \sigma_\bx)(\bomega),
    $$ 
    we have the following estimate in the same manner: 
    \begin{equation}
        \sup_{\bomega\in \overline{\mathcal{D}_{I'^d,\rho_1}}}|\partial_\bx^\bbeta(\sigma_\bx(\bomega)-\mathcal I_m^{\otimes d} \sigma_\bx(\bomega))|\leq C_{d,\rho,\rho_1} C_{\sigma,\bbeta}\,\theta^J\sqrt J,\quad\forall \bx\in I^d.
    \end{equation}
    
    By the Cauchy integral formula for derivatives, take $\eta_\sigma>0$ such that 
    $$ B_{\eta_\sigma}^{\otimes d}(0) := \underbrace{B_{\eta_\sigma}(0)\times\cdots\times B_{\eta_\sigma}(0)}_{d \text{ copies}}\subset \mathcal{D}_{I'^d,\rho_1},$$ we have 
    \begin{equation}\label{eqn:deriv-estimate}
        \begin{split}
            \frac{1}{\balpha!}\lt| \partial_\bx^\bbeta \lt[\partial_\bomega^\balpha(\sigma_\bx-\mathcal I_m^{\otimes d}\sigma_\bx)\bigg|_{\bomega=0}\rt] \rt| 
            &= \frac{1}{\balpha!}\lt| \partial_\bomega^\balpha\partial_\bx^\bbeta (\sigma_\bx-\mathcal I_m^{\otimes d}\sigma_\bx)\bigg|_{\bomega=0} \rt| \\
            &= \lt| \frac{1}{(2\pi i)^d}\oint_{|\omega_1|={\eta_\sigma}}\cdots\oint_{|\omega_d|={\eta_\sigma}} \frac{\partial_\bx^\bbeta (\sigma_\bx(\bomega)-\mathcal I_m^{\otimes d} \sigma_\bx(\bomega))}{\bomega^{\balpha+1}}\,d\omega_d\cdots d\omega_1 \rt| \\
            &\leq \frac{1}{(2\pi)^d}\frac{(2\pi \eta_\sigma)^d}{\eta_\sigma^{|\balpha|+d}} \max_{|\omega_k|=\eta_\sigma, k=1,\ldots,d} |\partial_\bx^\bbeta(\sigma_\bx(\bomega)-\mathcal I_m^{\otimes d} \sigma_\bx(\bomega))| \\
            &\leq C_{d,\rho,\rho_1}C_{\sigma,\bbeta}\,\eta_\sigma^{-|\balpha|}\theta^J\sqrt J,\quad\forall \bx\in I^d,
        \end{split}
    \end{equation}
    Thus, there exists
    $$\Psi_\balpha^{(n)}(\bx) = \frac{1}{\sigma^{(|\balpha|)}(0)}\sum_{1\leq j_1,\ldots,j_d\leq m} \partial_\bomega^\balpha l_\bj^{(m)}(0) \, \sigma_\bx(\bomega_\bj^{(m)}) \in V_n.$$
    for any $s\in\NN$, take all $|\bbeta|\leq s$ in \eqref{eqn:deriv-estimate} yields
    \begin{equation}\label{eqn:approx-to-poly-1}
        \lt\|\bx^\balpha - \Psi_\balpha^{(n)}(\bx) \rt\|_{W_s^\infty(I^d)} \lesssim \frac{\balpha!}{|\sigma^{(|\balpha|)}(0)|} \eta_\sigma^{-|\balpha|}\,\theta^J \sqrt J
        \lesssim \frac{\balpha!}{|\balpha|!} (\frac{\rho_\sigma}{\eta_\sigma})^{|\balpha|}\,\theta^J \sqrt J = \frac{\balpha!}{|\balpha|!}(\Gamma_\sigma)^{|\balpha|}\,\theta^J \sqrt J,
    \end{equation}
    where $\Gamma_\sigma = {\rho_\sigma} / \eta_\sigma $, and the hidden constant only depends on $\sigma,s,d,\rho,\rho_1$.

    \textbf{Step 2: Approximation to general Sobolev functions.}
    By the classical Jackson estimate for algebraic polynomials on a rectangle (see, e.g., \cite{devore1993constructive}), for any $f\in H^r(I^d)$ there exists a polynomial $f_\ell\in \mathcal{P}_\ell(I^d)$ of total degree at most $\ell$ such that 
    \begin{equation}\label{eqn:poly-approx-1}
        \|f-f_\ell\|_{H^s(I^d)}\lesssim \ell^{-(r-s)}\|f\|_{H^r(I^d)},\quad\forall\, 0\leq s\leq r.
    \end{equation}
    
    Take $\ell=\lceil\varepsilon m\rceil$, where $\varepsilon\in(0,1)$ is to be determined. 

    It remains to approximate the polynomial $f_\ell$. Consider the Legendre orthonormal polynomials $\{p_\balpha\}_{\balpha\geq 0}$ on $I^d$ defined in Corollary \ref{cor:rect-legendre-coeff-bound}.
    Suppose $f_\ell$ has the following Legendre polynomial expansion:
    $$ f_\ell(\bx) = \sum_{|\balpha|\leq \ell} \widehat{f_\ell}(\balpha)\, p_\balpha(\bx), $$
    satisfying 
    $$ \sum_{|\balpha|\leq \ell} |\widehat{f_\ell}(\balpha)|^2 = \|f_\ell\|_{L^2(I^d)}^2\leq 2\|f-f_\ell\|_{L^2(I^d)}^2 + 2\|f\|_{L^2(I^d)}^2 \lesssim \|f\|_{H^r(I^d)}^2.$$

    Suppose $$p_\balpha(\bx) = \sum_{0\leq\bbeta\leq\balpha} p_{\balpha,\bbeta} \,\bx^\bbeta,$$ and define
    $$ \widetilde{p}_\balpha^{(n)}(\bx) = \sum_{0\leq \bbeta \leq \balpha} p_{\balpha,\bbeta} \Psi_\bbeta^{(n)}(\bx).$$
    From Corollary \ref{cor:rect-legendre-coeff-bound} and \eqref{eqn:approx-to-poly-1} we have:
    \begin{equation}
    \begin{split}
        \|p_\balpha - \widetilde{p}_\balpha^{(n)}\|_{W^\infty_s(I^d)} 
        &\leq \sum_{0\leq \bbeta \leq \balpha} |p_{\balpha,\bbeta}| \lt\|\bx^\bbeta - \Psi_\bbeta^{(n)}(\bx)\rt\|_{W^\infty_s(I^d)} \\
        &\lesssim \sum_{0\leq \bbeta \leq \balpha} \Lambda_{I^d}^{|\balpha|}\, (\Gamma_\sigma)^{|\bbeta|}\,\theta^J \sqrt J 
        \lesssim \Upsilon^{|\balpha|} \theta^J\sqrt J,
    \end{split}
    \end{equation}
    where we can take $\Upsilon=\Lambda_{I^d}\max\{2,\Gamma_\sigma\}$. Define 
    $$\widetilde{f}_\ell^{(n)}(\bx) = \sum_{|\balpha|\leq \ell} \widehat{f_\ell}(\balpha) \,\widetilde p_\balpha^{(n)}(\bx)\in V_n,$$
    we can obtain its approximation error with $f_\ell$ is
    \begin{equation}
    \begin{split}
         \|f_\ell-\widetilde{f}_\ell^{(n)}\|_{H^s(I^d)} 
         &\lesssim \sum_{|\balpha|\leq \ell} |\widehat{f_\ell}(\balpha)|\,\|p_\balpha - \widetilde{p}_\balpha^{(n)}\|_{W^\infty_s(I^d)} \\
         &\lesssim \lt(\sum_{|\balpha|\leq \ell} |\widehat{f_\ell}(\balpha)|^2 \rt)^{1/2} \lt(\sum_{|\balpha|\leq \ell} \Upsilon^{2|\balpha|}\theta^{2J} J \rt)^{1/2} \\
         &\leq \binom{\ell+d}{d}^{1/2} \Upsilon^\ell \theta^J \sqrt J\, \|f\|_{H^r(I^d)} \\
         &\lesssim (\Upsilon^\varepsilon\theta^{C_2})^m \, m^{\frac{d+1}{2}}\, \|f\|_{H^r(I^d)}.
    \end{split}
    \end{equation}

    By choosing $\varepsilon$ sufficient small such that $\tilde\theta:=\Upsilon^\varepsilon\theta^{C_2} < 1$, we obtain 
    \begin{equation}\label{eqn:approx-to-fell}
        \|f_\ell-\widetilde{f}_\ell^{(n)}\|_{H^s(I^d)} \lesssim \tilde\theta^m\, m^{\frac{d+1}{2}}\,\|f\|_{H^r(I^d)}.
    \end{equation} 
    Combining \eqref{eqn:approx-to-fell} with \eqref{eqn:poly-approx-1}, we obtain $\widetilde{f}_\ell^{(n)}(\bx)\in V_n$, such that for any $s\in \NN$ and $0\leq s\leq r$, 
    $$\|f-\widetilde{f}_\ell^{(n)}\|_{H^s(I^d)}\leq \|f-f_\ell\|_{H^s(I^d)} + \|f_\ell -\widetilde{f}_\ell^{(n)}\|_{H^s(I^d)} \lesssim m^{-(r-s)}\|f\|_{H^r(I^d)}=n^{-\frac{r-s}{d}}\|f\|_{H^r(I^d)}. $$
    The construction above also gives the exponentially small polynomial-reproduction error at the neighboring integer Sobolev orders. Sobolev interpolation therefore yields the same estimate for fractional $s\in[0,r]$. Finally, restricting the approximant to $\Omega$ and using the boundedness of $E_\Omega$ gives \eqref{eqn:1d-analytic-approx} and completes the proof.
\end{proof}

\begin{remark}
    The assumption cannot hold for an entire function, because Cauchy's integral formula shows that the coefficients $\frac{1}{k!}\sigma^{(k)}(0)$ decay faster than any prescribed exponential rate:
    $$
    \lt|\frac{1}{k!}\sigma^{(k)}(0)\rt| = \lt|\frac{1}{2\pi i} \int_{|z|=R} \frac{\sigma(z)}{z^{k+1}}\,dz\rt| 
    \leq R^{-k}\max_{|z|=R} |\sigma(z)|, \quad \forall R>0.
    $$
\end{remark}

\subsection{Proof of Theorem \ref{thm_main_1d}}
\begin{proof}
    For any bounded interval $I$, denote $$l_I:=\inf I,\quad u_I:=\sup I.$$ 
    Without loss of generality, we may assume $x_0^{(m)}<\cdots<x_m^{(m)}$. 
    Since $\{x_i^{(m)}\}_{i=0}^m$ quasi-Chebyshev, by Lemma \ref{lem:transform-cheby} and \ref{lem:rescale-cheby}, we can write 
    $$e^{2x_i^{(m)}} = s_0^{(m)} + \Delta_m\delta_i^{(m)},\quad i=0,1,\ldots,m,$$ 
    where $\{\delta_i^{(m)}\}_{i=0}^m$ quasi-Chebyshev on $[0,1]$, $$\delta_0^{(m)}=0,\quad s_0^{(m)}=e^{2x_0^{(m)}}\to e^{2l_{I'}}, \quad\Delta_m = e^{2u_{I'}}-e^{2x_0^{(m)}}\to e^{2u_{I'}}-e^{2l_{I'}}.$$ 
    In fact, quasi-Chebyshev spacing at the endpoint gives $x_0^{(m)}-l_{I'}=O(m^{-2})$. Hence, with
    $$
        \Delta:=e^{2u_{I'}}-e^{2l_{I'}},
    $$
    we also have $0<\Delta_m\leq\Delta$ and $|\Delta_m-\Delta|=O(m^{-2})$.
    
    Since $\tanh(x-x_i^{(m)})$ can be written as
    $$\tanh(x-x_i^{(m)}) = \frac{e^{2x}-e^{2x_i^{(m)}}}{e^{2x}+e^{2x_i^{(m)}}}, $$
    we have

    $$\tanh(x-x_i^{(m)}) = \frac{e^{2x}-s_0^{(m)} - \Delta_m\delta_i^{(m)}}{e^{2x}+s_0^{(m)} + \Delta_m\delta_i^{(m)}}, $$
    therefore we can construct a $C^\infty$-diffeomorphism 
    \begin{align*}
        \mathcal{F}_m \colon I &\to I_m = \lt[\frac{1}{s_0^{(m)}+e^{2u_{I}}},\frac{1}{s_0^{(m)}+e^{2l_{I}}} \rt] \\
        x &\mapsto t = \frac{1}{s_0^{(m)}+e^{2x}}.
    \end{align*}
    The maps \(\mathcal F_m:I\to I_m\) are \(C^\infty\)-diffeomorphisms with uniformly bounded \(C^k\)-norms, as are their inverses, for every fixed \(k\geq 0\). The standard boundedness of pullbacks under uniformly smooth one-dimensional diffeomorphisms therefore makes the original approximation problem equivalent to the following one: 
    \begin{equation}
        \lt\|g - \sum_{i=0}^m a_i\phi_i^{(m)}\rt\|_{H^s(I_m)}\lesssim m^{-(r-s)} \|g\|_{H^r(I_m)}, \quad\forall g\in H^r(I_m),
    \end{equation}
    where
    $$\phi_i^{(m)}(t) = \frac{1-(2s_0^{(m)} + \Delta_m\delta_i^{(m)})t}{1+\Delta_m\delta_i^{(m)}t}.$$
    Put $s_\ast=e^{2l_{I'}}$. The limiting interval
    $$
        \lt[\frac{1}{s_\ast+e^{2u_I}},\frac{1}{s_\ast+e^{2l_I}}\rt]
    $$
    is compactly contained in $(0,1/s_\ast)$. Since $s_0^{(m)}\to s_\ast$, there are $m_0$ and a fixed interval $\tilde I$ such that, for every $m\geq m_0$,
    $$
        I_m\Subset\tilde I\Subset(0,\infty),
        \qquad
        \operatorname{dist}\lt(\tilde I,\lt\{0,\frac{1}{s_0^{(m)}}\rt\}\rt)\geq c>0.
    $$
    The intervals $I_m$ also have lengths bounded away from zero. By the Sobolev extension theorem \cite{adams2003sobolev}, the extension maps from $H^r(I_m)$ to $H^r(\tilde I)$ may therefore be chosen with norms bounded uniformly for $m\geq m_0$. It suffices to prove the following approximation result on $\tilde I$; the finitely many cases $m<m_0$ are absorbed by increasing the implicit constant:
    \begin{equation}\label{eqn:equivalent-approx}
        \inf_{g_m\in V_m} \lt\|g - g_m\rt\|_{H^s(\tilde I)}\lesssim m^{-(r-s)} \|g\|_{H^r(\tilde I)}, \quad\forall g\in H^r(\tilde I),
    \end{equation}
    where we denote $$V_m = {\rm span}\{\phi_i^{(m)}:i=0,1,\ldots,m\}.$$ 

    We can calculate that for all $i\neq 0$,  $$\phi_i^{(m)}-\phi_0^{(m)} = \frac{2\Delta_m\delta_i^{(m)}t (s_0^{(m)} t-1)}{1+\Delta_m\delta_i^{(m)}t} = \frac{2\Delta_m\delta_i^{(m)}}{1+\Delta_m\delta_i^{(m)}t}\beta_m(t),$$
    where $\beta_m(t)=t(s_0^{(m)} t-1)$. Therefore, 
    \begin{equation}\label{eqn:space-inclusion}
        V_m\supset \beta_m V_m',\quad\text{where }\, V_m' = {\rm span} \lt\{\frac{1}{1+\Delta_m\delta_i^{(m)}t}:i=1,2,\ldots,m\rt\}. 
    \end{equation}
    
    The separation property of $\tilde I$ shows that $\beta_m(t)$ and $\beta_m^{-1}(t)$ have uniformly bounded \(C^k\)-norms on $\tilde I$, for every fixed \(k\geq 0\). The Sobolev multiplier theorem therefore gives
    \begin{equation}\label{eqn:norm-equivalence}
        \|g\|_{H^s(\tilde I)} \simeq \|\beta_m g\|_{H^s(\tilde I)},\quad \forall g\in H^s(\tilde I)
    \end{equation}
    where the hidden constant only depends on $s$, independent of $m$.
    
    From \eqref{eqn:norm-equivalence} and \eqref{eqn:space-inclusion}, we obtain that, to prove \eqref{eqn:equivalent-approx}, it suffices to apply the following estimate to $g/\beta_m$:
    \begin{equation}\label{eqn:equivalent-approx-2}
        \inf_{g_m\in V_m'} \lt\|g - g_m\rt\|_{H^s(\tilde I)}\lesssim m^{-(r-s)} \|g\|_{H^r(\tilde I)}, \quad\forall g\in H^r(\tilde I).
    \end{equation}
    Take $\sigma(t) = \frac{1}{1+t}$ and $\{\omega_i^{(m)} = \Delta_m\delta_i^{(m)}\}_{i=1}^m$, which is quasi-Chebyshev on $\tilde I'=[0,\Delta]$ by Lemma \ref{lem:limit-cheby}. It is immediate that 
    $$\frac{1}{k!}|\sigma^{(k)}(0)| = 1,$$
    $$ 0\in \tilde I', \quad \{\omega x:\forall \omega\in \tilde I',x\in \tilde I\}\Subset \Omega_\sigma,$$
    here, for example, we can take
    $$\Omega_\sigma := \lt\{z: {\rm Re}\,z\in (-1/2,1+\Delta\sup\tilde I),{\rm Im}\,z\in (-1,1)\rt\}.$$
    Theorem \ref{thm_1d_analytic} then yields \eqref{eqn:equivalent-approx-2}, which completes the proof.
\end{proof}

\subsection{Proof of Theorem \ref{thm_assumption_general}}
    We begin with the asymptotic behavior of the sequence $\frac{1}{k!}\sigma^{(k)}(0)$. Suppose that a real-analytic function $\sigma$ has a pair of simple conjugate poles as its nearest singularities to the origin,
    \[
    p_\ast=\rho_\ast e^{i\theta_\ast},
    \qquad
    \overline{p_\ast}=\rho_\ast e^{-i\theta_\ast},
    \]
    and that there are no other singularities in \(|z|\leq R\), where \(R>\rho_\ast\). The conjugate symmetry follows from the reality of \(\sigma\) on the real axis.
    
    By Cauchy's integral formula and the residue theorem, we have
    \begin{equation}
    \begin{split}
        \frac{1}{k!} \sigma^{(k)}(0) &= -2{\rm Re}\bigg({\rm Res}(\sigma,\rho_\ast e^{i\theta_\ast})e^{-i(k+1)\theta_\ast}\bigg)\rho_\ast^{-(k+1)} + O(R^{-k})\\
        &=-2c_0\,\rho_\ast^{-(k+1)}\cos((k+1)\theta_\ast - \varphi) + O(R^{-k}),
    \end{split}
    \end{equation}
    where $c_0>0$ and the phase $\varphi\in\RR$ are defined by
    \[
    \operatorname{Res}(\sigma,p_\ast)=c_0 e^{i\varphi}.
    \]
    
    We now impose a quantitative non-cancellation condition on the oscillatory factor. Let
    \[
    \alpha=\frac{\theta_\ast}{\pi},
    \qquad
    \beta=\frac{\varphi}{\pi}-\frac12.
    \]
    \begin{definition}[Diophantine condition]\label{def:Diophantine}
        We say that $\alpha$ satisfies an \emph{Diophantine condition} with shift $\beta$ if there exist constants $\gamma>0$ and $\tau>0$ such that
        \begin{equation}\label{eqn:inhomo-Diophantine}
        \min_{\nu\in\mathbb Z}|m\alpha - \beta-\nu|
        \geq \gamma m^{-\tau},
        \qquad \forall m\in\mathbb N_+.
        \end{equation}
    \end{definition}
    We use the term in the standard sense of Diophantine approximation; see, for instance, \cite{Cassels1957,BugeaudLaurent2005}.
    
    If \eqref{eqn:inhomo-Diophantine} holds, then the cosine factor cannot approach its zeros too rapidly. Indeed, 
    \begin{equation}
    \begin{split}
        \rho_\ast^{-k-1}|\cos((k+1)\theta_\ast - \varphi)|
        &=\rho_\ast^{-k-1}\big|\sin\big(\pi((k+1)\alpha-\beta)\big)\big|\\
        &\geq 2\rho_\ast^{-k-1}\min_{\nu\in\ZZ}|(k+1)\alpha-\beta-\nu|\\
        &\gtrsim \rho_\ast^{-k-1}(k+1)^{-\tau}
        \gtrsim (\rho_\ast+\varepsilon)^{-k-1}.
    \end{split}
    \end{equation}
    By taking $\varepsilon$ small enough that $\rho_\ast+\varepsilon < R$, \eqref{eqn:derivative-bound} holds for all sufficiently large \(k\).

    \begin{proof}[Proof of Theorem \ref{thm_assumption_general}]
        We first prove that the Diophantine condition is generic under real shifts. Recall that
    $$
    \sigma_t(x)=\sigma(x+t),
    $$
    and for $t\in I$, where $I$ is a finite interval, the same conjugate pair of poles $p_\ast,\overline{p_\ast}$ of $\sigma$ gives the nearest singularities $p_\ast-t,\overline{p_\ast}-t$ of $\sigma_t$. Writing $p_\ast=a+ib$ with $b>0$, the corresponding pole of $\sigma_t$ is
    \[
    p_\ast-t=(a-t)+ib = 
    \rho_\ast(t)e^{i\theta_\ast(t)}.
    \]
    Thus
    \[
    \theta_\ast(t)=\arg((a-t)+ib),
    \]
    and
    \[
    \frac{d}{dt}\theta_\ast(t)
    =
    \frac{b}{(a-t)^2+b^2}
    > 0.
    \]
    Consequently, $t\mapsto \alpha(t)=\theta_\ast(t)/\pi$ is a $C^1$-diffeomorphism onto its image. For any fixed shift \(\beta\), the set of \(\alpha\) failing to satisfy a Diophantine condition has Lebesgue measure zero. Indeed, fix \(\tau>1\) and, on the bounded interval $\alpha(I)$, consider the sets
    \[
    E_m = 
    \left\{
    \alpha\in\alpha(I):\min_{\nu\in\mathbb Z}|m\alpha-\beta-\nu|<m^{-\tau}
    \right\}
    \]
    satisfy
    \[
    |E_m|\lesssim m^{-\tau},
    \]
    and hence \(\sum_m |E_m|<\infty\). By the Borel--Cantelli lemma, for almost every \(\alpha\), only finitely many of these inequalities occur. Excluding also the countable set where \(m\alpha-\beta\in\mathbb Z\) for some \(m\), we can choose \(\gamma>0\) so that
    \[
    \min_{\nu\in\mathbb Z}|m\alpha-\beta-\nu|
    \geq
    \gamma m^{-\tau},
    \qquad \forall m\geq 1.
    \]
    Therefore, for almost every \(t\in I\) (say for $t\in I_1$, with $I\setminus I_1$ of measure zero), the shifted function \(\sigma_t(x)=\sigma(x+t)\) has \(\theta_\ast(t)\) satisfying \eqref{eqn:inhomo-Diophantine}, which yields assumption \eqref{eqn:derivative-bound} for all sufficiently large $k$.

    It remains to show that, for almost every $t\in I_1$,
    $$\sigma_t^{(k)}(0)=\sigma^{(k)}(t)\neq 0,\quad \forall k\in\NN.$$
    Each $\sigma^{(k)}$ is analytic and nonzero as a function under the pole assumption, so its zeros on the real axis are isolated. Therefore
    $$ I_2 = \bigcup_{k=0}^\infty \{t\in I: \sigma^{(k)}(t) = 0\}$$
    is countable and has measure zero. For $t\in I_1\setminus I_2$, the large-$k$ lower bound and the finitely many nonzero low-order coefficients can be combined, after decreasing the constant and enlarging $\rho_t$ if necessary, to give \eqref{eqn:derivative-bound} for every $k\in\NN$. This proves Theorem \ref{thm_assumption_general}.
    \end{proof}

\subsection{Proof of Theorem \ref{thm_assumption_verify}}
    We use the following Diophantine property of algebraic numbers.
    \begin{lemma}[{\cite[Theorem 1.1]{ferreira2026arithmeticpropertiesargumentsalgebraic}}] \label{lem:algebraic-Diophantine}
        Let $e^{i\theta}$ be an algebraic number which is not a root of unity. Then $\theta/ 2\pi \in (-1/2,1/2]$ satisfies Diophantine condition with shift $0$.
    \end{lemma}

    \begin{proof}[Proof of Theorem \ref{thm_assumption_verify}]
    
    \begin{enumerate}
        \item $\displaystyle\sigma(x) = \tanh(x+\frac{q\pi}{2})$ for $\displaystyle q\in\QQ\backslash\{0,\pm 1\}$: the singularities of $\tanh(x)$ are $\displaystyle\pm \frac{\pi}{2}i + k\pi i$, $k\in\ZZ$. Therefore the nearest singularities of $\sigma(x)$ to origin are $\displaystyle p_\ast, \overline{p_\ast} = -\frac{q\pi}{2} \pm \frac{\pi}{2}i = \frac{\pi}{2}(-q\pm i)$ and ${\rm Res}(\sigma,p_\ast) = 1$ , which yields 
        $$ e^{i\theta_\ast} = \frac{-q + i}{\sqrt{q^2 + 1}}, \quad \varphi = 0$$
        which is an algebraic number. By Niven's Theorem \cite{Olmsted1945}, $e^{i\theta_\ast}$ is not a root of unity when $q\neq 0,\pm 1$, therefore by Lemma \ref{lem:algebraic-Diophantine}, $\theta_\ast/\pi$ satisfies Diophantine condition with shift $0$, i.e. there exists $\gamma>0,\tau>0$ such that
        \begin{equation}
        \min_{\nu\in\mathbb Z}|m\frac{\theta_\ast}{\pi}-\nu|
        \geq \gamma m^{-\tau},
        \qquad \forall m\in\mathbb N_+.
        \end{equation}
        Therefore we have 
        \begin{equation}
        \begin{split}
            \min_{\nu\in\mathbb Z}|m\frac{\theta_\ast}{\pi}-(\frac{\varphi}{\pi}-\frac{1}{2})-\nu|
            &\geq \frac{1}{2}\min_{\nu\in\mathbb Z}|2m\frac{\theta_\ast}{\pi}+1-2\nu| \\
            &\geq \frac{\gamma}{2} (2m)^{-\tau} = \frac{\gamma}{2^{1+\tau}}m^{-\tau},
        \qquad \forall m\in\mathbb N_+,
        \end{split}
        \end{equation}
        which verifies \eqref{eqn:inhomo-Diophantine}. Therefore it remains to verify that $\sigma^{(k)}(0)\neq 0$ for all $k\in\NN$. We record the following simple consequence of the Gelfond--Schneider theorem. 
        Indeed, we first show that \(\tanh(q\pi/2)\) is transcendental. Choose the branch of the logarithm for which
        \[
        \log(-1)=\pi i.
        \]
        Then
        \[
        e^{q\pi}
        =
        \exp\bigl((-iq)\log(-1)\bigr)
        \]
        is one value of \((-1)^{-iq}\). Since \(-1\) is algebraic and different from \(0,1\), while \(-iq\) is an algebraic irrational number (here complex numbers are not regarded as rational when they have an imaginary part not equal to 0), the Gelfond--Schneider theorem implies that
        \(
        e^{q\pi}
        \)
        is transcendental. On the other hand,
        \[
        \tanh(\frac{q\pi}{2})=\frac{e^{q\pi}-1}{e^{q\pi}+1}.
        \]
        If \(\tanh(\frac{q\pi}{2})\) were algebraic, then
        \[
        e^{q\pi}
        =
        \frac{1+\tanh(\frac{q\pi}{2})}{1-\tanh(\frac{q\pi}{2})}
        \]
        would also be algebraic, a contradiction. Hence \(\tanh(\frac{q\pi}{2})\) is transcendental.
        
        Next, define polynomials \(P_k\in\mathbb{Z}[z]\) by
        \[
        P_0(z)=z,\qquad
        P_{k+1}(z)=(1-z^2)P_k'(z).
        \]
        An induction gives
        \[
        \frac{d^k}{dx^k}\tanh x=P_k(\tanh x).
        \]
        Moreover, each \(P_k\) is a nonzero polynomial; in fact,
        \[
        \deg P_k=k+1
        \]
        and its leading coefficient is \((-1)^k k!\).
        Since \(P_k\) is a nonzero polynomial with integer coefficients and \(\tanh(q\pi/2)\) is transcendental, we have
        \[
        \sigma^{(k)}(0) = P_k(\tanh(\frac{q\pi}{2}))\neq 0.
        \]
        The verification of \eqref{eqn:derivative-bound} for $\sigma(x)$ is then completed.
    \item $\displaystyle\sigma(x) = \frac{1}{1+e^{-x+\pi q}}$ for $q\in\QQ\backslash\{0,\pm 1\}$: it follows from the identity 
    $$
    \sigma(x) = \frac{1}{1+e^{-x+ q\pi}} = \frac{1}{2}+ \frac{1}{2}\tanh\lt(\frac{x-q\pi}{2}\rt).
    $$
    Indeed, the translation in item 1 applies with $-q$, and the additional scaling of the input multiplies the $k$-th Taylor coefficient by $2^{-k}$; this only changes the exponential constant in \eqref{eqn:derivative-bound}.
    \item $\displaystyle\sigma(x) = \frac{c_1 + c_2 x}{1+x^2}$ for $c_1,c_2\neq 0$: verification of \eqref{eqn:derivative-bound} can be done by directly calculating its Taylor expansion,
        $$ 
        \sigma(x) = \sum_{\ell=0}^\infty (-1)^\ell (c_1 x^{2\ell} + c_2 x^{2\ell+1}).
        $$
    \item $\displaystyle\sigma(x) = \frac{1}{1+(x+q)^2}$ for $q\in\QQ\backslash\{0,\pm 1\}$: since 
        $$
        \sigma(x) = \frac{1}{2i}\lt( \frac{1}{x+q-i}-\frac{1}{x+q+i} \rt),
        $$
        we obtain the only singularities of $\sigma$ are $$p_\ast, \overline{p_\ast} = -q \pm i = \rho_\ast e^{\pm i\theta_\ast},$$
        where $\rho_\ast = \sqrt{q^2 + 1}$, $e^{i\theta_\ast} = \displaystyle\frac{-q + i}{\sqrt{q^2 + 1}}$,
        $${\rm Res}(\sigma,p_\ast) = -\frac{i}{2} = c_0 e^{i\varphi},$$
        where $c_0 = \frac{1}{2}$ and $\varphi = -\frac{\pi}{2}$. 
        By Cauchy's integral formula and residue theorem, we have
        $$\frac{1}{k!}\sigma^{(k)}(0) = \rho_\ast^{-(k+1)}\sin((k+1)\theta_\ast).$$
        From the verification of $\tanh(x+\frac{q\pi}{2})$, the Diophantine property of $\theta_\ast/\pi$ has been illustrated, and note that there is no low order term $O(R^{-k})$ here, therefore \eqref{eqn:derivative-bound} in this case is proved. 
    \item $\displaystyle\sigma(x) = \arctan(x+q)$ for $q\in\QQ\backslash\{0,\pm 1\}$: use
    $$ \arctan(x+q) = \arctan q + \int_0^x \frac{1}{1+(t+q)^2}\,dt.$$
    Thus, for $k\geq1$, its $k$-th Taylor coefficient is the $(k-1)$-st Taylor coefficient from item 4 divided by $k$. The factor $k^{-1}$ is absorbed by replacing the exponential radius with any slightly larger one; the zeroth coefficient is nonzero because $q\neq0$.
    \end{enumerate}
    \end{proof}

\section{Conclusion}\label{sec:conclusion}
We have reduced the construction of sharp linearized shallow-network spaces on general domains to a one-dimensional approximation problem. Quasi-Chebyshev parameter sets with resolution $m$ yield the sharp one-dimensional Sobolev order for the analytic activations covered by \eqref{eqn:derivative-bound}, and a rational change of variables gives the corresponding result for arbitrary quasi-Chebyshev translates of $\tanh$. The lifting theorem in \cite{petrushev1998approximation}, in the quasi-uniform form proved in Appendix~\ref{app:quasiuniform-lifting}, then produces the optimal multidimensional $L^2$ order for product-type parameter sets. With $n=m^d$, these spaces contain at most order $n$ features and attain the rate $m^{-r}=n^{-r/d}$. The direct construction of tensorized quasi-Chebyshev parameter set supplies the analogous $H^r$-to-$H^s$ estimate for the general analytic class. The fixed-interval parameter geometry avoids the factorial-scale clustering present in the finite-difference construction in \cite{mhaskar1996neural}; quantitative conditioning estimates for the resulting feature matrices remain an important question.

\section{Appendix: Ridge lifting with quasi-uniform directions}\label{app:quasiuniform-lifting}

This appendix explains why the cubature points used in \cite{petrushev1998approximation} can be replaced by an arbitrary quasi-uniform family of directions. We verify only the two properties of the directions required in the lifting argument. Once these properties are established, the remaining approximation argument follows directly from \cite[Sections~3, 7, and~8]{petrushev1998approximation}.

Let $d\geq2$, and let $\Omega_m\subset\SS^{d-1}$ be an arbitrary quasi-uniform family satisfying
$$
|\Omega_m|\simeq m^{d-1}.
$$
Its mesh norm and separation are both comparable to $m^{-1}$, with constants independent of $m$. We denote by $\mathbb P_L(\SS^{d-1})$ the space of spherical polynomials of degree at most $L$.

\begin{lemma}\label{lem:quasi-uniform-cubature-lifting}
There exists a constant $a>0$ such that, for every $m$, there are positive weights $\lambda_{\omega,m}$, $\omega\in\Omega_m$, satisfying
\begin{equation}\label{eqn:quasi-uniform-cubature-lifting}
\int_{\SS^{d-1}}S(\xi)\,d\xi
=
\sum_{\omega\in\Omega_m}\lambda_{\omega,m}S(\omega),
\qquad
S\in\mathbb P_{\lfloor am\rfloor}(\SS^{d-1}),
\end{equation}
and
\begin{equation}\label{eqn:quasi-uniform-weight-lifting}
\lambda_{\omega,m}\simeq m^{-(d-1)},
\qquad
\omega\in\Omega_m.
\end{equation}
The constants may depend on $d$ and on the uniform bound for the mesh ratios, but not on $m$ or $\omega$.
\end{lemma}

\begin{proof}
Since the mesh norm of $\Omega_m$ is comparable to $m^{-1}$ and its mesh ratio is uniformly bounded, \cite[Corollary~4.4]{narcowich2006localized} applies with $L=\lfloor am\rfloor$ once $a>0$ is chosen sufficiently small. It gives a positive cubature formula exact on $\mathbb P_L(\SS^{d-1})$. Moreover, the lower bound for the weights is comparable to the $(d-1)$st power of the mesh norm, while the upper bound is comparable to $L^{-(d-1)}$. This proves \eqref{eqn:quasi-uniform-weight-lifting}. The corresponding statement on $\SS^1$ is the classical positive trigonometric quadrature result; see also \cite{mhaskar2001spherical}. A related Marcinkiewicz--Zygmund formulation for arbitrary scattered points is given in \cite[Theorem~2.1]{leGiaMhaskar2009localized}.
\end{proof}

The cubature identity is sufficient for the low-frequency part of the argument. For the high-frequency part, one also needs the discrete synthesis estimate corresponding to \cite[Lemmas~4.2--4.4]{petrushev1998approximation}. We first record the required kernel bound. For every $\ell\geq1$ and $\eta\in\SS^{d-1}$,
\begin{equation}\label{eqn:quasi-uniform-kernel-sum}
\sum_{\omega\in\Omega_m}
\lambda_{\omega,m}\ell^{d-1}
\bigl(1+\ell\rho(\omega,\eta)\bigr)^{-d}
\lesssim
1+\left(\frac{\ell}{m}\right)^{d-1}.
\end{equation}
Indeed, divide the sphere into the cap $\rho(\omega,\eta)<m^{-1}$ and the annuli
$$
\frac{j}{m}\leq\rho(\omega,\eta)<\frac{j+1}{m},
\qquad j\geq1.
$$
The separation of $\Omega_m$ implies that the first cap contains at most a constant number of points and that the $j$th annulus contains at most $C(j+1)^{d-2}$ points. Using \eqref{eqn:quasi-uniform-weight-lifting}, the left-hand side of \eqref{eqn:quasi-uniform-kernel-sum} is therefore bounded by
$$
C\left(\frac{\ell}{m}\right)^{d-1}
+
C\left(\frac{\ell}{m}\right)^{d-1}
\sum_{j=1}^{\infty}
j^{d-2}
\left(1+\frac{j\ell}{m}\right)^{-d}.
$$
Splitting the sum at $j=m/\ell$ proves \eqref{eqn:quasi-uniform-kernel-sum}.

After a harmless fixed rescaling of $\ell$, let $W_\ell$ be a localized kernel reproducing $\mathbb P_\ell(\SS^{d-1})$ as in \cite[Proposition~4.1]{petrushev1998approximation}. Its localization estimate is
$$
|W_\ell(\xi\cdot\eta)|
\lesssim
\ell^{d-1}\bigl(1+\ell\rho(\xi,\eta)\bigr)^{-d}.
$$
The reproducing identity, the Cauchy--Schwarz inequality, and \eqref{eqn:quasi-uniform-kernel-sum} give
\begin{equation}\label{eqn:quasi-uniform-mz-lifting}
\sum_{\omega\in\Omega_m}
\lambda_{\omega,m}|S(\omega)|^2
\lesssim
\left[1+\left(\frac{\ell}{m}\right)^{d-1}\right]
\|S\|_{L^2(\SS^{d-1})}^2,
\qquad
S\in\mathbb P_\ell(\SS^{d-1}).
\end{equation}
The spherical reproducing identity and another application of the Cauchy--Schwarz inequality then give the synthesis estimate in \cite[Lemma~4.4]{petrushev1998approximation}. Thus both the low-frequency cubature identities and the high-frequency stability estimate remain valid for arbitrary quasi-uniform directions.

\begin{theorem}[Quasi-uniform ridge lifting]\label{thm:quasi-uniform-ridge-lifting}
Let $\Omega\subset\RR^d$ be a bounded domain satisfying the hypotheses of \cite[Theorem~8.2]{petrushev1998approximation}, and let $I\subset\RR$ be the corresponding fixed projection interval. Suppose that $q>0$ and that the univariate spaces $X_m$, with $\dim X_m\lesssim m$, satisfy
\begin{equation}\label{eqn:univariate-assumption-appendix}
\inf_{v\in X_m}\|u-v\|_{L^2(I)}
\lesssim
m^{-q}\|u\|_{H^q(I)}.
\end{equation}
For any quasi-uniform family $\Omega_m\subset\SS^{d-1}$ satisfying $|\Omega_m|\simeq m^{d-1}$, define
$$
Y_m
:=
\operatorname{span}
\{v(\omega\cdot\circ):v\in X_m,\ \omega\in\Omega_m\}.
$$
Then, with
$$
r=q+\frac{d-1}{2},
$$
we have
\begin{equation}\label{eqn:quasi-uniform-lifting-rate}
\inf_{g\in Y_m}\|f-g\|_{L^2(\Omega)}
\lesssim
m^{-r}\|f\|_{H^r(\Omega)},
\qquad
f\in H^r(\Omega).
\end{equation}
Moreover, setting $n=m^d$,
$$
\dim Y_m\lesssim m^d=n.
$$
\end{theorem}

\begin{proof}
We only describe the modification concerning the directions. Choose $M=\lfloor cm\rfloor$, where $c>0$ is sufficiently small that
$$
2M\leq\lfloor am\rfloor.
$$
By \eqref{eqn:quasi-uniform-cubature-lifting}, the cubature formula is exact for all products of spherical polynomials occurring in the ridge decomposition through degree $M$. Consequently, the discrete representation and Parseval identities used in \cite[(4.22)--(4.23)]{petrushev1998approximation} remain valid.

The directions enter the proof of \cite[Theorem~7.1]{petrushev1998approximation} only through these cubature identities and the synthesis estimate corresponding to \cite[Lemmas~4.2--4.4]{petrushev1998approximation}. Replacing those ingredients by \eqref{eqn:quasi-uniform-cubature-lifting} and \eqref{eqn:quasi-uniform-mz-lifting}, respectively, leaves the rest of the proof unchanged. Since $M\simeq m$, the argument gives
$$
\inf_{g\in Y_m}\|f-g\|_{L^2(\Omega)}
\lesssim
M^{-q-\frac{d-1}{2}}
\|f\|_{H^{q+\frac{d-1}{2}}(\Omega)}
\lesssim
m^{-r}\|f\|_{H^r(\Omega)}.
$$
The passage from the unweighted estimate \eqref{eqn:univariate-assumption-appendix} to the corresponding weighted univariate approximation, as well as the passage to the general domain $\Omega$, is exactly the argument in \cite[Section~8 and Theorem~8.2]{petrushev1998approximation} and is independent of the particular choice of directions. Finally,
$$
\dim Y_m
\leq
(\dim X_m)|\Omega_m|
\lesssim
m\,m^{d-1}
=
m^d=n.
$$
\end{proof}

\bibliographystyle{abbrv}
	\bibliography{ref}
\end{document}